\documentclass[11pt]{amsart}
\usepackage{txfonts}
\usepackage{amscd,amssymb,amsmath,amsbsy,amsthm}
\usepackage[colorlinks,plainpages,backref,urlcolor=blue,breaklinks]{hyperref}
\usepackage{tikz}
\usetikzlibrary{cd,calc,intersections}
\usetikzlibrary{svg.path}
\usepackage{enumitem}
\usepackage{nccmath, relsize}
\usepackage{xfrac}
\usepackage{microtype}

\newtheorem{theorem}{Theorem}[section]
\newtheorem{lemma}[theorem]{Lemma}
\newtheorem{corollary}[theorem]{Corollary}
\newtheorem{proposition}[theorem]{Proposition}

\theoremstyle{definition}
\newtheorem{definition}[theorem]{Definition}
\newtheorem{example}[theorem]{Example}
\newtheorem{remark}[theorem]{Remark}
\newtheorem{question}[theorem]{Question}

\newtheorem*{ack}{Acknowledgment}

\newcommand{\Z}{\mathbb{Z}}
\newcommand{\Q}{\mathbb{Q}}
\newcommand{\R}{\mathbb{R}}
\newcommand{\C}{\mathbb{C}}

\newcommand{\K}{\mathbb{K}}
\newcommand{\LL}{\mathbb{L}}

\newcommand{\T}{\mathbb{T}}

\renewcommand{\u}{\mathbf{u}}
\newcommand{\x}{\mathbf{x}}
\newcommand{\y}{\mathbf{y}}
\newcommand{\bt}{\mathbf{t}}

\renewcommand{\k}{\Bbbk}
\newcommand{\RR}{\mathcal{R}}
\newcommand{\VV}{\mathcal{V}}
\newcommand{\A}{{\mathcal{A}}}
\newcommand{\B}{{\mathcal{B}}}

\newcommand{\wV}{\mathcal{W}}

\newcommand{\XX}{{\mathcal X}}

\DeclareMathOperator{\rank}{rank}

\DeclareMathOperator{\im}{im}

\DeclareMathOperator{\codim}{codim}

\DeclareMathOperator{\id}{id}
\DeclareMathOperator{\ab}{{ab}}

\DeclareMathOperator{\supp}{supp}

\DeclareMathOperator{\Hom}{{Hom}}
\DeclareMathOperator{\Tor}{{Tor}}

\DeclareMathOperator{\ann}{{ann}}

\DeclareMathOperator{\Spec}{{Spec}}

\DeclareMathOperator{\init}{in}

\DeclareMathOperator{\Tors}{Tors}

\DeclareMathOperator{\TC}{TC}
\DeclareMathOperator{\Trop}{Trop}
\DeclareMathOperator{\trop}{trop}
\DeclareMathOperator{\orb}{orb}
\DeclareMathOperator{\lcm}{{lcm}}

\DeclareMathOperator{\tor}{Tors}
\DeclareMathOperator{\val}{val}

\DeclareMathOperator{\Newt}{Newt}
\DeclareMathOperator{\Cay}{Cay}
\DeclareMathOperator{\conv}{conv}

\newcommand{\wX}{\widetilde{X}}
\DeclareMathOperator{\alg}{alg-gp}

\newcommand{\surj}{\twoheadrightarrow}
\newcommand{\inj}{\hookrightarrow}
\newcommand{\isom}{\xrightarrow{
   \,\smash{\raisebox{-0.6ex}{\ensuremath{\scriptstyle\simeq}}}\,}}
\newcommand{\compl}{\mathrm{c}}

\newcommand{\abs}[1]{\left| #1 \right|}

\newcommand{\parr}[1]{(\!( #1 )\!)}

\def\set#1{{\{ #1\}}}

\newcommand{\myempty}{\text{\O}}

\def\bra#1{\{\!\{#1\}\!\}}

\definecolor{lime}{HTML}{A6CE39}
\DeclareRobustCommand{\orcidicon}{
	\begin{tikzpicture}
	\draw[lime, fill=lime] (0,0) 
	circle [radius=0.16] 
	node[white] {{\fontfamily{qag}\selectfont \tiny ID}};
	\draw[white, fill=white] (-0.0625,0.095) 
	circle [radius=0.007];
	\end{tikzpicture}
	\hspace{-2mm}
}
\foreach \x in {A, ..., Z}{\expandafter\xdef\csname orcid\x\endcsname{\noexpand\href{%
https://orcid.org/\csname orcidauthor\x\endcsname}{\noexpand\orcidicon}}}

\title[Sigma-invariants and tropical varieties]%
{Sigma-invariants and tropical varieties}

\author[Alexander~I.~Suciu]{Alexander~I.~Suciu$^1$\orcidA{}}
\address{Department of Mathematics,
Northeastern University,
Boston, MA 02115, USA}
\email{\href{mailto:a.suciu@northeastern.edu}{a.suciu@northeastern.edu}}
\urladdr{\href{http://web.northeastern.edu/suciu/}%
{web.northeastern.edu/suciu/}}
\thanks{$^1$Supported in part by the Simons Foundation Collaboration 
Grants for Mathematicians \#354156 and \#693825}



\date{March 25, 2021}

\begin{document}

\begin{abstract}
The Bieri--Neumann--Strebel--Renz invariants $\Sigma^q(X,\Z)\subset H^1(X,\R)$ 
of a connected, finite-type CW-complex $X$ are the vanishing loci for 
Novikov--Sikorav homology  in degrees up to $q$, while the characteristic varieties 
$\VV^q(X) \subset H^1(X,\C^{\times})$ are the nonvanishing loci for homology 
with coefficients in rank $1$ local systems in degree $q$. We show that each 
BNSR invariant $\Sigma^q(X,\Z)$ is contained in the complement of the tropical 
variety associated to the algebraic variety $\VV^{\le q}(X)$, and provide 
applications to several classes of groups and spaces.
\end{abstract}

\maketitle

\section{Introduction}
\label{sect:intro}

\subsection{The Sigma-invariants}
\label{subsec:bnsr-intro}
In \cite{BNS}, Bieri, Neumann, and Strebel introduced powerful new invariants 
of finitely generated groups.  To each such group $G$ they associated 
a subset  $\Sigma^1(G)$ of the unit sphere $S(G)$ in the real vector space 
$\Hom(G,\R)$. This ``geometric'' invariant of the group is cut out of the 
sphere by open cones, and is independent of a finite generating set 
for $G$.  In \cite{BR}, Bieri and Renz recast the BNS invariant in homological  
terms, and defined a nested family of higher-order invariants, 
$\{\Sigma^q(G,\Z)\}_{q\ge 1}$, which record the finiteness properties 
of normal subgroups of $G$ with abelian quotients. 

In \cite{FGS}, Farber, Geoghegan, and Sch\"utz further extended these 
notions, from groups to spaces. The BNSR invariants of a connected, 
finite-type CW-complex $X$ form a nested sequence of subsets,  
$\{\Sigma^q(X,\Z)\}_{q\ge 1}$, inside the unit sphere $S(X)\subset H^1(X,\R)$. 
The sphere $S(X)$ can be thought of as parametrizing 
all free abelian covers of $X$, while the $\Sigma$-invariants 
(which are again open subsets), keep track of the geometric 
finiteness properties of those covers. 
The BNSR invariants $\Sigma^q(X,\Z)$ are the vanishing 
loci for Novikov--Sikorav homology  in degrees up to $q$. 
The significance of these invariants lies in the fact that they 
control the finiteness properties of kernels of projections to 
abelian quotients; in particular, they encode information about 
algebraic fiberings of the group $G=\pi_1(X)$, i.e., homomorphisms 
$G \surj \Z$ with finitely generated kernels.

The actual computation of the BNS invariant is enormously complicated, 
and has been achieved so far only for some special classes of groups, 
such as metabelian groups \cite{BGr}, one-relator groups \cite{Br}, 
right-angled Artin groups \cite{MMV}, K\"{a}hler groups \cite{De10}, 
and the pure braid groups \cite{KMM}.  If $G$ is the fundamental 
group of a compact, connected $3$-manifold $M$, the BNS invariant 
is the projection onto the unit sphere of the fibered faces 
of the unit ball in Thurston's norm, introduced in \cite{Th} 
in order to understand all the ways in which $M$ may fiber over 
the circle.  In all these examples, and also for the groups studied by 
Kielak in \cite{Ki19}, the set $\Sigma^1(G)$ is the intersection 
of $S(G)$ with a finite union of open rational polyhedral cones. 
As shown in \cite{BNS}, though, this is not always the case.

It thus makes sense to look for approximations to the BNSR invariants 
which (1) are more computable, and (2) are finite unions of open rational 
polyhedral cones. In a special case, such an approximation 
was found by McMullen \cite{McM}, who showed that, for a 
compact, connected, orientable $3$-manifold with empty or 
toroidal boundary, the Thurston norm unit ball $B_T$ is contained 
in the Alexander norm unit ball $B_A$, which is defined in terms of the 
Newton polytope of the multi-variable Alexander polynomial $\Delta_G$. 
Using the characteristic varieties (a generalization of the zero-locus of $\Delta_G$), 
Papadima and the author gave in \cite{PS-plms} a much more general upper 
bound---valid for all BNSR invariants $\Sigma^q(X,\Z)$---albeit not 
as sharp as McMullen's in the case of $3$-manifolds. Our goal here 
is to strengthen the bound from \cite{PS-plms} so as to largely recover 
McMullen's result (at least at the level of the fibered faces of $B_T$), and 
also to recast Delzant's theorem from  \cite{De10} in this general framework.  
The tools we use come mostly from the theory of cohomology jump loci, 
and draw in an essential way on ideas and methods from tropical geometry. 

\subsection{Tropicalized characteristic varieties}
\label{subsec:exptcone-intro}

Let $X$ be a space as above, and let $G=\pi_1(X)$ be its fundamental group. 
The group of complex-valued characters, 
$\T_G \coloneqq \Hom(G,\C^{\times})=H^1(X,\C^{\times})$, 
is a complex algebraic group, which may be thought of as the moduli space of 
rank $1$ local systems on $X$.  Taking homology with coefficients in such local 
systems carves out subvarieties $\VV^q(X)\subset \T_G$ where the homology in 
degree $q$ does not vanish. These {\em characteristic varieties}, which control 
the Betti numbers of regular abelian covers of $X$ (see 
e.g.~\cite{Su-pisa12, Su-imrn, SYZ-pisa}), provide the initial input towards 
finding computable bounds for the BNSR invariants. 

For the next step, we enlarge our moduli space to $\K$-valued characters, 
where $\K=\C\bra{t}$ is the field of Puiseux series with complex 
coefficients. Notably, this is an algebraically closed  field, which 
comes endowed with a non-Archimedean valuation, $\K^{\times}\to \Q\subset \R$, 
see e.g.~\cite{MS}.  Let $\nu_X\colon H^1(X,\K^{\times}) \to H^1(X,\R)$  be the 
homomorphism induced by the coefficient map.  Given a subvariety $W$ of the 
algebraic $\K$-group $H^1(X,\K^{\times})$, we define its  {\em tropicalization}, 
$\Trop(W)$, to be the closure of $\nu_X(W\times_{\C} \K)$ inside the 
Euclidean space $H^1(X,\R)$.
  
Our main result (proved in Theorem \ref{thm:bns-trop}), reads as follows.

\begin{theorem}
\label{thm:bns-trop-intro}
Let $\VV^{\le q}(X)$ be the union of the characteristic 
varieties of $X$ in degrees up to $q$. Then 
$\Sigma^q(X,\Z) \subseteq S\big(\!\Trop( \VV^{\le q}(X))\big)^{\compl}$.
\end{theorem}

This upper bound is a finite union of open polyhedral cones 
in $H^1(X,\R)$, intersected with the unit sphere $S(X)$. 
The theorem recovers---in a much stronger form---the 
main result of \cite{PS-plms}, which asserts that 
$\Sigma^q(X,\Z)$ is contained in $S( \tau^{\R}_1(\VV^{\le q}(X)))^{\compl}$, 
where $\tau_1(W)\subset \C^n$ is the {\em exponential tangent cone}\/ at the 
identity to a subvariety $W\subset (\C^{\times})^n$, and $\tau^{\R}_1(W)$ 
are the real points on it.  Indeed, as we show in Proposition \ref{prop:tc trop}, 
we always have an inclusion $\tau_1^{\R}(W)\subseteq \Trop(W)$. Nevertheless, 
this inclusion is oftentimes strict, since, for instance, $\tau_1(W)$ is a union of 
rationally defined linear subspaces, whereas $\Trop(W)$ is a (not necessarily 
symmetric about the origin) polyhedral complex. Furthermore,  tropicalization 
may detect components of $\VV^{q}(X)$ that do not pass through the origin, 
or even do not lie in $\T_G^0$ (the identity component of $\T_G$), 
whereas the exponential tangent cone only depends on the analytic 
germ at $1$ of $\VV^{q}(X)$. 

\subsection{Bounding the BNS invariant}
\label{subsec:BNSbound-intro}

For a finitely generated group $G$, Theorem \ref{thm:bns-trop-intro} yields 
the following tropical bound on the BNS invariant:
\begin{equation}
\label{eq:inc-G-intro}
\Sigma^1(G) \subseteq -S(\Trop( \VV^{1}(G)))^{\compl},
\end{equation}
where $-\Sigma$ denotes the image of a subset $\Sigma\subset S(G)$ 
under the antipodal map. 
The minus sign arises due to the conventions we follow here, according 
to which $\Sigma^1(G)=-\Sigma^1(G,\Z)$, as in \cite{BR}; for a detailed account 
of our choices of signs in several definitions, we refer to Remark \ref{rem:signs}. 
We single out two classes of groups where the bound from \eqref{eq:inc-G-intro} 
can be refined.

The first class (which includes the $3$-manifold groups mentioned 
previously) consists of groups $G$ for which the Alexander polynomial 
$\Delta_G$ is symmetric and satisfies a certain condition which 
guarantees that $\VV^1(G)\cap \T_G^0$ coincides (at least away 
from $1$) with the zero-locus of $\Delta_G$. 
For such groups, we prove in Theorem \ref{thm:sigma1-G} that 
\begin{equation}
\label{eq:sigma-facets-intro}
\Sigma^1(G) \subseteq \bigcup_{\text{$F$ an open facet of $B_A$}} S(F)\, .
\end{equation} 
For closed, orientable, $3$-manifolds with empty or toroidal boundary 
this inclusion also follows from the aforementioned 
work of Thurston, Bieri--Neumann--Strebel, and McMullen.

The second class (which includes K\"ahler groups, arrangement groups, 
and certain Seifert manifold groups), consists of groups $G$ for which 
there are homomorphisms $f_\alpha$ from $G$ onto groups 
$G_{\alpha}$ so that $\VV^1(G_\alpha)$ contains a component 
of $\T_{G_\alpha}$. In Theorem \ref{thm:sigma-pencils}, we prove that 
\begin{equation}
\label{eq:sigma-pen-intro}
\Sigma^1(G) \subseteq \bigg(\bigcup_{\alpha} 
S\big(f^*_{\alpha}(H^1(G_\alpha,\R))\big)\bigg)^{\compl}.
\end{equation}
When $G=\pi_1(M)$ is the fundamental group of a compact K\"ahler 
manifold and the maps $f_{\alpha}$ are induced by suitable orbifold  
fibrations from $M$ to orbifold Riemann surfaces, work of Delzant  \cite{De10} 
shows that this inclusion holds as equality, and so 
$\Sigma^1(M)=S\big(\!\Trop(\VV^1(M))\big)^{\compl}$. 
When $M$ is the complement of an arrangement of hyperplanes in 
$\C^{\ell}$ and $G=\pi_1(M)$, the maps $f_{\alpha}$ 
arise from orbifold fibrations with base a punctured, genus $0$ orbifold. 
The bound \eqref{eq:sigma-pen-intro} takes into account the translated 
tori in $\VV^1(M)$---which the previous bounds did not---but it remains to be 
seen whether the bound is sharp in this setting.

\subsection{Organization}
\label{subsec:org}

The paper is organized in two, roughly equal parts.  Part one  
develops the general theory relating the Sigma-invariants to the 
tropicalization of the characteristic varieties.  
In \S\ref{sect:trop} we review some basic concepts from tropical geometry, 
and establish a connection with the construction of the exponential tangent cone.  
In \S\ref{sect:cvs} we recall needed facts about the characteristic varieties 
and the Alexander polynomial, while in \S\ref{sect:trop-cv} we define and study the 
tropicalization of the characteristic varieties. After a quick review of the 
BNSR invariants in \S\ref{sec:bnsr}, we prove in \S\ref{sect:sig-trop-char} 
our main result. Finally, in \S\ref{sect:bns-trop-bound} we establish formulas 
\eqref{eq:sigma-facets-intro} and \eqref{eq:sigma-pen-intro}. 

The second part outlines applications of this theory to several classes of 
spaces and groups. We start in \S\ref{sect:one-rel} with $1$-relator groups, 
where Brown's algorithm provides a quick way for computing the $\Sigma$-invariants, 
and continue in \S\ref{sect:3mfd} with $3$-manifold groups.  In preparation for 
the final three sections, we discuss in \S\ref{sect:orbifolds} the cohomology 
jump loci and the $\Sigma$-invariants of $2$-dimensional orbifolds.  
In \S\ref{sect:seifert} we identify a large class of Seifert fibered manifolds for 
which the BNS invariant is empty. We conclude in \S\S\ref{sect:kahler}--\ref{sect:arrs} 
with the interplay between the tropical characteristic varieties and the BNS invariants 
of compact K\"ahler manifolds and complements of hyperplane arrangements.

\section{Tropical varieties and exponential tangent cones}
\label{sect:trop}

We start by reviewing the basics of tropical geometry, following 
the treatment in the monograph of Maclagan and Sturmfels \cite{MS},  
as well as \cite{BK, CTY, EKL, OS, Pa09, Ra12}. 

\subsection{Valuations}
\label{subsec:val}
Let $\K\coloneqq \C\bra{t}=\bigcup_{n\ge 1} \C\parr{t^{1/n}}$ be the field 
of Puiseux series with complex coefficients. A nonzero element of $\K$ 
has the form $c(t)=c_1t^{a_1} + c_2t^{a_2}  + \cdots$, 
where $c_i\in \C^{\times}$ and $a_1<a_2<\cdots$ are rational 
numbers that have a common denominator.
  
The field $\K$ is algebraically closed, and admits 
a valuation, $\val\colon \K^{\times}\to \Q$, given by $\val(c(t)) =a_1$. 
This makes $\K$ into a valued field, with value group $\Q$, 
and defines a non-Archimedean absolute value on $\K$, by 
setting $\abs{c}=\exp(-\val(c))$ for $c\ne 0$ and $\abs{0}=0$. 
The valuation ring, $R= \{x\in \K\mid \val(x)\ge 0\}$, 
is equal to $\bigcup_{n\ge 1} \C\llbracket{t^{1/n}}\rrbracket$; 
this is a local ring with maximal ideal 
$\mathfrak{m}=\{x\in \K\mid \val(x)> 0\}$ 
and residue field  $R/\mathfrak{m}=\C$. 

For an element $h\in R$, we denote by 
$\left.h\right|_{t=0}$ its image in the residue field. 
A Laurent polynomial $f\in \K[\x^{\pm}] \coloneqq \K[x_1^{\pm 1},\dots,x_n^{\pm 1}]$ 
has the form  $f=\sum_{\u\in A} a_\u \x^\u$, with $a_\u\in \K^{\times}$, for 
some finite subset $A=\supp(f)\subset \Z^n$.  For $w\in \Q^n$, write 
$f(t^{w_1}x_1,\dots,t^{w_n}x_n)=t^{b} g(x_1,\dots,x_n)$, 
where $g\in R[x_1^{\pm 1},\dots,x_n^{\pm 1}]$ 
and no positive power of $t$ divides $g$.  The initial form of $f$ 
with respect to $w$ is then given by 
$\init_{w}{f}=\left.g\right|_{t=0}$. 
  
\subsection{Tropical varieties}
\label{subsec:trop}

Let $(\K^{\times})^n=\Spec(\K[\x^{\pm}])$ be the algebraic 
torus of dimension $n$ over $\K$.  Taking the $n$-fold product 
of the valuation map $\K^{\times}\to \Q$, we obtain 
the map $\val\colon (\K^{\times})^n\to \Q^n\subset\R^n$.
For an ideal $I\subset \K[\x^{\pm}]$, let $W=V(I)$ be the subvariety of 
$(\K^{\times})^n$ cut out by $I$.  The {\em tropicalization}\/ of 
$W \subset (\K^{\times})^n$ is the closure in $\R^n$ (with the 
Euclidean metric topology) of the image of $W$ under the valuation map, 
\begin{equation}
\label{eq:trop-var}
\Trop(W)\coloneqq \overline{\val(W)}\, .
\end{equation}  
It follows directly from the definition that 
\begin{equation}
\label{eq:trop-union}
\Trop(V\cup W)=\Trop(V) \cup \Trop(W)\, .
\end{equation}

An important result in tropical geometry is that a point $w\in \Q^n$ 
belongs to $\Trop(W)$ if and only if $V(\init_{w}{I})\ne \emptyset$, 
where  $\init_{w}{I}$ is the ideal of $\C[\x^{\pm}]$ spanned by 
$\{ \init_{w}{f} \mid f\in I \}$, see ~\cite{Pa09, BK, MS}. Moreover, 
as shown in \cite[Theorem 7.11]{Ra12} in a more general context, 
\begin{equation}
\label{eq:trop-rat}
\Trop(W)\cap \Q^n = \val(W)\, .
\end{equation}
 
The tropicalization of a hypersurface in $(\K^{\times})^n$ can be 
described as follows.  For a Laurent polynomial 
$f=\sum_{\u\in A} a_\u \bt^\u$, 
one defines the {\em tropical polynomial}\/ $\trop(f)$ by
\begin{equation}
\label{eq:trop poly}
\trop(f)(w)=\min_{\u\in A} \left(w\cdot\u+\val(a_\u)\right) .
\end{equation}
This is a piecewise linear, concave function $\R^n \to \R$. 
As a consequence of Kapranov's theorem \cite[Thm.~2.1.1]{EKL}, 
the set $\Trop(V(f))$ coincides with the corner locus of $\trop(f)$: 
that is, the subset of $\R^n$ on which the minimum is achieved 
for at least two values of $\u$.  In other words, the tropicalization 
of $V(f)$ is the locus in $\R^n$ where  $\trop(f)$ fails to be linear.

As shown in \cite[Prop. 3.1.6]{MS}, the tropical hypersurface 
$\Trop(V(f))$ is the support of a pure, rational polyhedral complex 
of dimension $n-1$ in $\R^n$. This complex may be described 
as the $(n-1)$-skeleton of the
polyhedral complex dual to a regular subdivision of the Newton polytope of
$f$ given by the weights $\val(a_\u)$ on the lattice points in 
$\Newt(f)\coloneqq \conv\{\u \mid a_\u\ne 0\}\subset \R^n$.

Finally, if $W=V(I)$ is an arbitrary subvariety in $(\K^{\times})^n$, then 
$\Trop(W)=\bigcap_{f\in I}\Trop(V(f))$. As shown in \cite[Cor. 3.2.4]{MS}, 
this set is the support of a rational polyhedral complex in $\R^n$, a fact 
which explains once again why all the rationals points on $\Trop(W)$ 
are contained in $\val(W)$, as stated in \eqref{eq:trop-rat}.
For instance, if $W$ is a curve, then $\Trop(W)$ is a 
graph with rational edge directions.

Tropicalization enjoys the following naturality property with respect to morphisms 
of algebraic tori. Let $\alpha\colon (\K^{\times})^m \to (\K^{\times})^n$ 
be a monomial map, so that the induced morphism on coordinate rings, 
$\alpha^*\colon \K[\y^{\pm}]\to \K[\x^{\pm}]$ is given by the $n\times m$ 
integral matrix $A$ associated to the homomorphism
\begin{equation}
\label{eq:alpha-vee}
\begin{tikzcd}[column sep=16pt]
\alpha^{\vee}\colon \Hom_{\alg}(\K^{\times},( \K^{\times})^m)=\Z^m \ar[r]& 
 \Hom_{\alg}(\K^{\times},( \K^{\times})^n)=\Z^n .
 \end{tikzcd}
\end{equation}
Let $\trop(\alpha)\coloneqq\alpha^{\vee}\otimes \id_{\R}\colon \R^m \to \R^n$ 
be the $\R$-linear map given by the transpose matrix $A^{\top}$.  
Then, as shown in \cite[Cor.~3.2.13]{MS}, 
for every subvariety $W\subset (\K^{\times})^m$, 
\begin{equation}
\label{eq:trop-mat}
\Trop\big(\overline{\alpha(W)}\big)=\trop(\alpha)\big(\!\Trop(W)\big)\, .
\end{equation}

\begin{remark}
\label{rem:trop-prod-int}
If $V$ and $W$ are irreducible subvarieties of $\C^n$, then 
$\Trop(V\times W)=\Trop(V) \times \Trop(W)$, see  \cite{CTY}.  
On the other hand, tropicalization 
does not always respect intersections:  if $V$ and $W$ are subvarieties of 
$(\K^{\times})^n$, then $\Trop(V\cap W)\subseteq \Trop(V) \cap \Trop(W)$, 
but the inclusion may be strict. Nevertheless, as shown in \cite{OS}, 
this inclusion holds as equality, provided that the tropicalizations 
intersect in the expected dimension.
\end{remark}

\subsection{Tropicalizing subvarieties in a complex torus}
\label{subsec:trop-C}
The ``constant coefficient case" is that of varieties defined 
over the field $\C$ with trivial valuation.  Let  $W$ be a 
subvariety of $(\C^{\times})^n$---also known as a complex, very 
affine variety---and let $W\times_{\C} \K$
be the subvariety of $(\K^{\times})^n$ obtained by extension of 
the base field. The tropicalization of $W$ is then defined as 
\begin{equation}
\label{eq:trop-const}
\Trop(W)\coloneqq\Trop(W\times_{\C} \K)\, .
\end{equation}
For such varieties, the tropicalization is a rational polyhedral 
fan in $\R^n$. When $W=V(f)$ is a hypersurface in $(\C^{\times})^n$ 
defined by a Laurent polynomial $f\in\C[\x^{\pm }]$, 
there is a very concrete, geometric interpretation of $\Trop(W)$.

Given a polytope $P$, we denote by $\mathcal{F}(P)$ its face fan, 
i.e., the set of cones spanned by the faces of $P$, and by 
$\mathcal{N}(P)$ its (inner) normal fan.  If $0$ is in the interior 
of $P$, then the normal fan to $P$ coincides with  
the face fan of the (polar) dual polytope,
\begin{equation}
\label{eq:polarity}
\mathcal{N}(P) = \mathcal{F}(P^{\Delta})\, ,
\end{equation}
see \cite[Exercise~7.1]{Ziegler}.  As in \S\ref{subsec:trop}, let $\Newt(f)$ 
be the Newton polytope of $f$, that is, the convex hull in $\R^n$ of the 
set $\supp(f)\subset \Z^n$.  Then, as noted in \cite{BK},
\begin{equation}
\label{eq:normal-newt}
\Trop(V(f))=\mathcal{N}(\Newt(f))^{\codim >0},
\end{equation}
the positive-codimension skeleton of the normal fan to that polytope. 
In particular, a top-dimen\-sional cone in $\Trop(V(f))$
corresponds to an edge of $\Newt(f)$, or to a ridge in 
the normal fan.

\subsection{Tropicalized translated subtori}
\label{subsec:trop tori}
Let $\k=\C$ or $\K$, and let $(\k^{\times})^n$ be an algebraic torus over $\k$.  
By definition, an {\em algebraic subtorus}\/ of $(\k^{\times})^n$ is the image, $T=\im(\alpha)$, of 
a monomial inclusion, $\alpha\colon (\k^{\times})^{m}\hookrightarrow (\k^{\times})^{n}$, 
for which the image of the dual map, $\alpha^{\vee}\colon \Z^m\inj \Z^n$, 
is a primitive sublattice of $\Z^n$; see for instance \cite{BK,SYZ-jpaa}.  
It follows from formula \eqref{eq:trop-mat}  (see also \cite{BK}) that the 
tropicalization of such an algebraic subtorus is the $m$-dimensional 
linear subspace 
\begin{equation}
\label{eq:tropt}
\Trop(T)=\im(\trop(\alpha))= \alpha^{\vee} (\Z^m) \otimes \R \subset \R^n  .
\end{equation}

A {\em translated subtorus}\/ in $(\k^{\times})^n$ is a subvariety of the 
form $z\cdot T$, where $T$ is an algebraic subtorus and 
$z\in(\k^{\times})^n$. The tropicalization of such a variety 
is an affine subspace, given by
\begin{equation}
\label{eq:troptw}
\Trop(z\cdot T)=\Trop(T)+\val(z)\, .
\end{equation}
In the constant coefficient case ($\k=\C$), 
this formula reduces to
\begin{equation}
\label{eq:tropzt}
\Trop(z\cdot T)=\Trop(T)\, .
\end{equation}

\begin{remark}
\label{rem:bk}
A partial converse to formula \eqref{eq:tropzt} was established in 
\cite[Lem.~3.1]{BK}:  Let $W\subset (\C^{\times})^n$ be an irreducible 
subvariety, let $T$ be an algebraic subtorus, and suppose   
$\Trop(W)\subset \Trop(T)$; then there exists a point 
$z\in(\C^{\times})^n$ such that  $W\subset z\cdot T$. 
\end{remark}

\subsection{The exponential map}
\label{subsec:exp}
Let $\exp\colon \C^n \to (\C^{\times})^n$ be the $n$-fold product 
of the map $\C\to \C^{\times}$, $z\mapsto e^z$.  
The next lemma is based on work from \cite{Su-imrn, SYZ-jpaa, SYZ-pisa}; 
since we will need the explicit construction from the lemma, 
we give a quick, mostly self-contained proof.

\begin{lemma}
\label{lem:subtori}
If $L$ is a linear subspace of $\Q^n$, then $\exp(L\otimes_{\Q} \C)$ is 
an algebraic subtorus of $(\C^{\times})^n$. 
Conversely, every algebraic subtorus $T\subset (\C^{\times})^n$ 
can be realized as $T=\exp(L\otimes_{\Q}  \C)$, 
for some linear subspace $L\subset \Q^n$.
\end{lemma}

\begin{proof}
Let $\Lambda=L\cap \Z^n$; then $\Lambda$ is a primitive sublattice 
of $\Z^n$ and $L=\Lambda\otimes \Q$. The inclusion map, 
$\Lambda \inj \Z^n$, is the Pontryagin dual to 
a monomial inclusion map, $T \inj (\C^{\times})^n$,
where $T$ is a complex algebraic subtorus. 
Moreover, as shown in \cite[Lemma 6.1]{SYZ-jpaa}, 
\begin{equation}
\label{eq:torus}
T=\exp(L\otimes_{\Q}  \C)\, .
\end{equation}

Now suppose $T$ is an algebraic subtorus of $(\C^{\times})^n$. 
If $\alpha\colon T \inj (\C^{\times})^n$ is the corresponding 
monomial inclusion map, and if we let $L=\im(\alpha^{\vee})\otimes \Q$,  
then $T=\exp(L\otimes_{\Q}  \C)$, by \eqref{eq:torus}.
\end{proof}

As an application, we obtain the following corollary.

\begin{corollary}
\label{cor:trop trans}
Let $V$ be a subvariety of $(\C^{\times})^n$, all of whose irreducible 
components are translated subtori.  Write 
$V=\bigcup_{i} \rho_i \cdot  \exp(L_i\otimes_{\Q} \C)$,   
for some (finitely many) linear subspaces $L_i\subset \Q^n$ 
and some $\rho_i\in (\C^{\times})^n$. 
Then $\Trop (V) = \bigcup_i L_i\otimes_{\Q}  \R$. 
\end{corollary}

\begin{proof}
By Lemma \ref{lem:subtori}, it is indeed possible to write each irreducible 
component of $V$ in the stated form. We then have
\begin{gather}
\begin{aligned}
\Trop (V) &=  
\Trop\Big( \bigcup\nolimits_{i} \rho_i \cdot  \exp(L_i\otimes_{\Q} \C)\Big)
&
\\
 &= \bigcup\nolimits_{i}  \Trop\big( \rho_i \cdot  \exp(L_i\otimes_{\Q} \C)\big)
 &&\qquad \text{by \eqref{eq:trop-union}} \\
 &= \bigcup\nolimits_{i}  \Trop\big(\!\exp(L_i\otimes_{\Q} \C)\big)
 &&\qquad \text{by \eqref{eq:tropzt}}\\
 &= \bigcup\nolimits_{i}  L_i\otimes_{\Q} \R 
 &&\qquad \text{by \eqref{eq:tropt},}
\end{aligned}
 \end{gather}
 where at the last step we also used the proof of Lemma \ref{lem:subtori}. 
This completes the proof.
\end{proof}

\subsection{The exponential tangent cone}
\label{subsec:exptcone}

The following notion was introduced by 
Dimca, Papadima, and the author in \cite{DPS-duke}, and 
further studied in \cite{PS-plms, Su-imrn, SYZ-pisa}.  

\begin{definition}
\label{def:exp tcone}
Let $W\subset (\C^{\times})^n$ be an algebraic subvariety. 
The {\em exponential tangent cone}\/ of $W$ at $1$ is 
the homogeneous subvariety $\tau_1(W)\subset \C^n$, 
defined by 
\begin{equation}
\label{eq:taudef}
\tau_1(W)= \big\{ z\in \C^n \mid \exp(\lambda z)\in W,\ 
\text{for all $\lambda\in \C$} \big\} \, .
\end{equation}
\end{definition}

This set depends only on the analytic germ of $W$ at the identity;  
in particular, $\tau_1(W)\ne \emptyset$ if and only if $1\in W$.  
Furthermore, $\tau_1$ commutes with finite unions, as  
well as arbitrary intersections.  The main features 
of this construction are encapsulated in the following result.

\begin{lemma}[\cite{DPS-duke,Su-imrn}]
\label{lem:tc-linear}
For a subvariety $W\subset (\C^{\times})^n$,  
the set $\tau_1(W)$ is a finite union of rationally 
defined linear subspaces of $\C^n$.  Moreover, 
$\tau_1(W)$  is contained in $\TC_1(W)$, the 
tangent cone at $1$ to $W$. 
\end{lemma}

Let $\tau_1^{\R}(W)\coloneqq \tau_1(W)\cap \R^n$ be the set of real 
points on the exponential tangent cone. The next proposition---%
which may be thought of as the starting point of our investigation---%
establishes a noteworthy relationship between this set and the 
tropicalization of $W$. 

\begin{proposition}
\label{prop:tc trop}
Let $W\subset (\C^{\times})^n$ be an algebraic subvariety. 
Then $\tau_1^{\R}(W)\subseteq \Trop(W)$.
\end{proposition}

\begin{proof}
As a consequence of Lemma \ref{lem:tc-linear}, the set $\tau^{\R}_1(W)$ is 
a finite union of rationally defined linear subspaces of $\R^n$. 
Let $L\otimes_{\Q} \R$ be one of those subspaces, with $L$ 
as its rational points.

By the definition of $\tau_1(W)$, the set $T=\exp(L\otimes_{\Q} \C)$  
lies inside $W$; thus, $\Trop(T)\subset \Trop(W)$. 
On the other hand, by Corollary \ref{cor:trop trans}, $\Trop(T)=L\otimes_{\Q} \R$.
This shows that $L\otimes_{\Q} \R \subset \Trop(W)$, thereby proving the claim.
\end{proof}

If $W$ is an algebraic subtorus of $(\C^{\times})^n$, the above inclusion is 
of course an equality.  In general, though, the inclusion is strict.  
For instance, if $W=\rho T$, with $T$ a subtorus and $\rho \notin T$, 
then $\tau_1^{\R}(W)=\myempty$, whereas $\Trop(W)=\Trop(T)$ is 
nonempty. More generally, we have the following result.

\begin{proposition}
\label{prop:t1-trop}
Let $W$ be an algebraic subvariety of $(\C^{\times})^n$. 
Suppose there is a subtorus $T\subset (\C^{\times})^n$ such that 
$T\not\subset W$, yet $\rho T\subset W$ for some $\rho\in (\C^{\times})^n$. 
Then $\tau^{\R}_1(W)\subsetneqq \Trop(W)$.
\end{proposition}

\begin{proof}
By Lemma \ref{lem:subtori}, there is a linear subspace $L\subset \Q^n$ 
such that $T=\exp(L\otimes_\Q \C)$. By Corollary \ref{cor:trop trans}, 
$\Trop (\rho T)=L\otimes_\Q \R$.  Since, by assumption, 
$\rho T\subset W$, we must have
\begin{equation}
\label{eq:lqr-y}
L\otimes_\Q \R\subset \Trop(W)\, .
\end{equation}

On the other hand, the assumption that $T\not\subset W$ 
together with definition \eqref{eq:taudef} imply that 
$L\otimes_\Q \C\not\subset \tau_1(W)$; thus,
\begin{equation}
\label{eq:lqr-n}
L\otimes_\Q \R\not\subset \tau^{\R}_1(W)\, .
\end{equation}
Putting together \eqref{eq:lqr-y} and  \eqref{eq:lqr-n}, we conclude that 
the set $\Trop(W) \setminus \tau^{\R}_1(W)$ is nonempty.
\end{proof}

\section{Characteristic varieties and Alexander polynomials}
\label{sect:cvs}

We give in this section a brief review of the theory of 
characteristic varieties and Alexander polynomials.  
For details and further references we refer to 
\cite{DPS-duke, PS-plms, PS-mrl, Su-imrn, Su-tc3d}  for the 
former, and \cite{DPS-imrn, EN, Hi97, McM, Tu} for the latter.

\subsection{Jump loci for twisted homology}
\label{subsec:jumps} 
Let $X$ be a connected CW-complex with finite $q$-skeleton, 
for some $q\ge 1$.  Without loss of generality, we may assume $X$ 
has a single $0$-cell, call it $x_0$.  Let $G=\pi_1(X,x_0)$ 
be the fundamental group of $X$ based at $x_0$, and let 
\begin{equation}
\label{eq:char-gp}
\Hom(G,\C^{\times})=H^1(X,\C^{\times})
\end{equation}
be the group of complex-valued, multiplicative characters of $G$, 
which we shall denote at times as $\T_X$ or $\T_G$.  
Since $\C^{\times}$ is abelian, the abelianization map, 
$\ab\colon G\surj G_{\ab}$, induces an isomorphism, 
$\ab^*\colon \T_{G_{\ab}}\isom \T_G$.  The character 
group of $G$ is a complex algebraic group, with coordinate ring  
$\C[G_{\ab}]$, and with identity $1$ corresponding to the trivial representation. 
The identity component, $\T_G^{0}$, is an algebraic torus of 
dimension $n=b_1(X)$;  the connected components of $\T_G$ 
are translates of this torus by characters indexed by the torsion 
subgroup of $G_{\ab}=H_1(X,\Z)$. 

For each character $\rho\colon G\to \C^{\times}$, we let 
$\C_{\rho}$ be the corresponding rank $1$ local system on $X$, 
i.e., the vector space $\C$, viewed as a module over 
$\C[G]$ via the action $g\cdot a=\rho(g)a$. The {\em characteristic 
varieties}\/ of $X$ (in degree $i\le q$) are the jump loci 
for homology with such twisted coefficients, 
\begin{equation}
\label{eq:cvx}
\VV^i(X)=\big\{\rho \in H^1(X,\C^{\times}) \mid  
H_i(X, \C_{\rho}) \ne 0\big\}\, .
\end{equation}

Here is an alternative description of these  loci, 
which makes it clear that they are Zariski closed 
subsets of the character group, at least for $i<q$.  
Let $X^{\ab}\to X$ be the maximal abelian cover, corresponding  
to the projection $\ab\colon G\surj G_{\ab}$.  Upon lifting the cell 
structure of $X$ to this cover, we obtain a chain complex of free 
$\Z[G_{\ab}]$-modules, $(C_{*}(X^{\ab},\Z), \partial^{\ab})$.
By definition, a character $\rho\in \T_{X}=\T_{G_{\ab}}$ belongs to $\VV^i(X)$ 
precisely when $\rank \partial^{\ab}_{i+1}(\rho) + 
\rank \partial^{\ab}_{i}(\rho) < c_i$, 
where $c_i$ is the number of $i$-cells of $X$ 
and the evaluation of $\partial^{\ab}_i$ at $\rho$ is obtained by applying 
the ring morphism $\C[G_{\ab}]\to \C$, $g\mapsto \rho(g)$ to each entry.  
Hence, $\VV^i(X)$ is the zero-set of the ideal of 
minors of size $c_i$ of the block-matrix 
$\partial^{\ab}_{i+1} \oplus \partial^{\ab}_{i}$.  
The case $i=q$ is more delicate, requiring the replacement 
of $X$ with a CW-complex having finite $(q+1)$-skeleton, 
yet the same jump loci; see \cite[Prop.~4.1]{PS-mrl}.  

We can also define the characteristic varieties  of 
$X$ with coefficients in an arbitrary field $\k$ 
as the subvarieties $\VV^i(X,\k)$  of the algebraic 
group  $H^1(X,\k^{\times})$ 
consisting of those characters $\rho\colon \pi_1(X)\to \k^{\times}$ 
for which $H^i(X,\k_{\rho})\ne 0$. The argument outlined above 
shows that the jump loci $\VV^i(X,\k)$ are determinantal varieties 
of matrices defined over $\Z$. Consequently, these constructions 
are compatible with restriction and extension of the base field; 
that is, if $\k\subset \LL$ is a field extension, then 
\begin{align}
\label{eq:cv-ext}
\VV^i(X,\k)&=\VV^i(X,\LL) \cap H^1(X,\k^{\times})\, ,
\\[2pt]
\label{eq:cv-base}
\VV^i(X,\LL) &= \VV^i(X,\k) \times_{\k} \LL\, .
\end{align}

\subsection{Properties of the characteristic varieties}
\label{subsec:prop-cv}
The characteristic varieties of $X$ are homotopy-type invariants.  Indeed, 
if $f\colon X\to Y$ is a homotopy equivalence, then the induced 
morphism on character groups, $f^*\colon H^1(Y,\C^{\times})\to 
H^1(X,\C^{\times})$, restricts to isomorphisms 
$ \VV^i(Y)\isom \VV^i(X)$.

Clearly, $1\in \VV^i(X)$ if and only if $b_i(X)\ne0 $; moreover, 
$\VV^0(X)= \{1 \}$.  The  set $\VV^1(X)$ depends only on the 
fundamental group $G=\pi_1(X)$, 
and, in fact, only on its maximal metabelian quotient, $G/G''$; thus, 
we shall sometimes write this set as $\VV^1(G)$.  We also 
have the following (partial) functoriality property.  

\begin{proposition}[\cite{Su-imrn}]  
\label{prop:v1-nat}
Let  $G$ be a finitely generated group, 
and let $\varphi\colon G\surj Q$ be a surjective homomorphism.   
Then the induced morphism between character groups, 
$\varphi^*\colon \T_Q\to \T_G$, 
$\varphi^* (\rho)(g)=\varphi(\rho(g))$, is injective, and 
 restricts to an embedding $\VV^1(Q) \inj \VV^1(G)$.
\end{proposition}

Each homology group $H_i(X^{\ab},\C)$ naturally acquires the 
structure of a $\C[G_{\ab}]$-module. The {\em Alexander varieties}\/ 
of $X$ are the zero-sets of the annihilators of these modules,
\begin{equation}
\label{eq:alex-var}
\wV^i(X)=V(\ann(H_i(X^{\ab},\C)))\, .
\end{equation}
As such, these algebraic sets are subvarieties of the character group 
$\T_G=\Spec(\C[G_{\ab}])$. 
We will  mainly be interested here in the unions of the characteristic varieties 
up to a fixed degree, $\VV^{\le q}(X)=\bigcup_{i\le q} \VV^i(X)$. 
If $X$ has finite $q$-skeleton, then 
\begin{equation}
\label{eq:cv-alex}
\VV^{\le q}(X) = \wV^{\le q}(X)\, ,
\end{equation}
where $\wV^{\le q}(X)=\bigcup_{i\le q} \wV^i(X)$, 
see \cite{PS-plms,PS-mrl}.  Moreover, if $b_1(X)>0$, then 
$\VV^{\le 1}(X)=\VV^{1}(X)$.

\subsection{The Alexander polynomial}
\label{subsec:alex-poly}
 
Let $H=G_{\ab}/\tor(G_{\ab})$ be the maximal torsion-free 
abelian quotient of the group $G=\pi_1(X,x_0)$.  It is readily 
seen that the group ring $\Z{H}$ is a commutative Noetherian 
ring and a unique factorization domain.  
Let $q\colon X^{H}\to X$ be the regular cover corresponding to the 
projection $G\surj H$, i.e., the maximal torsion-free abelian cover of $X$.  
The {\em Alexander module}\/ of $X$ is 
defined as the relative homology group 
\begin{equation}
\label{eq:alex-mod}
A_X= H_1(X^H,q^{-1}(x_0), \Z)\, ,
\end{equation}
viewed as a $\Z{H}$-module. This module depends only on the group $G$: 
if $I_G=\ker\, (\varepsilon \colon \Z{G}\to \Z)$ is the augmentation ideal, 
then $A_X\cong A_G$, where $A_G\coloneqq \Z{H}\otimes_{\Z{G}} I_G$. 
To see why, fix a basepoint $\tilde{x}_0\in  q^{-1}(x_0)$;  sending each element  
$g-1\in I_G$ to the path in $X^H$ from $\tilde{x}_0$ to $g\tilde{x}_0$ obtained by 
lifting the loop $g$ at $\tilde{x}_0$ induces an isomorphism 
$A_G \isom A_X$. 

Now let $E_1(A_X)\subseteq \Z{H}$ be the 
ideal of codimension $1$ minors in a $\Z{H}$-presentation for $A_X$.  
The {\em Alexander polynomial}\/ of $X$ is then defined as the greatest 
common divisor of the elements in this determinantal ideal,
\begin{equation}
\label{eq:alex-poly}
\Delta_X=\gcd ( E_1(A_X))\, .
\end{equation}

When $G$ admits a finite presentation, say, 
$G=\langle  x_1,\dots ,x_{m}\mid r_1,\dots ,r_{s}\rangle$, 
we can be more explicit. Let $J_G$ be the $s\times m$ matrix 
whose entries are the Fox derivatives $\partial_j r_i$ of the relators, viewed as 
elements in $\Z{G}$.  Then $\VV^1(G)$ coincides (away from $1$) with 
the subvariety of $\T_G$ defined by the minors of size $m-1$ of $J_G^{\ab}$, 
the matrix obtained from $J_G$ by applying the morphism 
$\ab\colon \Z{G}\to \Z{G_{\ab}}$ to its entries.  Likewise, $\Delta_X$ is 
the greatest common divisor of the minors of size $m-1$ of $J_G^{\alpha}$, 
where $\alpha\colon \Z{G}\to \Z{H}$ is defined by the projection map $G\surj H$. 

Set $n=b_1(G)$ and assume that $n>0$. 
Upon fixing a basis for $H\cong \Z^n$, we may identify 
$\T_G^0=\T_H$ with $(\C^{\times})^n$ and  $\Z{H}$ with the ring 
of Laurent polynomials in $t_1^{\pm1},\dots , t_n^{\pm 1}$. 
The Alexander polynomial $\Delta_{G}\in \Z{H}$ is  
well-defined up to multiplication by units in this ring, 
which are all of the form $\pm g$, for some $g\in H$ 
(we write $\Delta \doteq \Delta'$ if $\Delta = \pm g\cdot \Delta'$ 
in $\Z{H}$).

Let $\Newt(\Delta_G)\subset H_1(G,\R)=\R^n$ be the Newton polytope 
of $\Delta_{G}$. Every cohomology class $\phi \in H^1(G;\Z)\cong \Hom(H,\Z)$  
defines a linear functional, $\phi\colon \R^n \to \R$; we let $\phi(\Newt(\Delta_G))$ 
be the image of the Newton polytope under this map. In \cite{McM}, 
McMullen defined the {\em Alexander norm}\/ of a class $\phi \in H^1(G;\Z)$, 
denoted $\|\phi\|_A$, as the length of the interval $\phi(\Newt(\Delta_G))\subset \R$. 
Clearly, the function $\phi\mapsto \|\phi\|_A$ is convex and linear on rays, making it 
a semi-norm on $H^1(G;\Z)$, which extends to a semi-norm on $H^1(G,\R)$.  
We let $B_A\subset H^1(G,\R)$ be the unit ball in this semi-norm. 

If the Alexander polynomial of $G$ is symmetric, i.e., invariant up to units in 
$\Z{H}$ under the involution $g\mapsto g^{-1}$,  then the Alexander norm ball, 
$B_A$, is, up to a scale factor of $1/2$, the polar dual of the Newton polytope 
of $\Delta_G$ (see \cite{Lg} for a proof):
\begin{equation}
\label{eq:alex-ball}
B_A=\tfrac{1}{2} \Newt(\Delta_G)^{\Delta}.
\end{equation}

We conclude with a result that relates $V(\Delta_G)$, the algebraic 
hypersurface in $\T_G^0$ defined by the vanishing of $\Delta_G$, 
to the first characteristic variety of $G$. This is done under a certain 
hypothesis relating the ideal of $\Z{H}$ generated by $\Delta_G$, 
the augmentation ideal $I_H\subset \Z{H}$, and the Alexander module, 
$A_G=\Z{H}\otimes_{\Z{G}} I_G$.

\begin{proposition}[\cite{DPS-imrn}]
\label{prop:zz1-pi}
Suppose that $I^p_H\cdot ( \Delta_{G} ) =E_1(A_G)$, for some $p\ge 0$. 
Then $\VV^1(G)\cap \T_G^0 = V(\Delta_{G}) \cup \set{1}$. 
\end{proposition}

Groups with first Betti number equal to $1$, finitely 
presented groups  with more generators than relators (such as the $1$-relator 
groups from \S\ref{sect:one-rel}), and the $3$-manifold groups 
from \S\ref{sect:3mfd} all satisfy the hypothesis of Proposition \ref{prop:zz1-pi}. 

Now suppose $G$ satisfies this hypothesis, and also 
$G_{\ab}$ is torsion-free (so that $G_{\ab}=H$).  Then 
$\VV^1 (G)$ itself coincides with $V(\Delta_G)$, 
at least away from $1$: for instance, 
if $\Delta_G \doteq 1$ then  $1\in \VV^1 (G)$, though 
$V(\Delta_G)=\myempty$. 
On the other hand, if $\Delta_G(1)=0$, then $\VV^1 (G)=V(\Delta_G)$. 

\section{Tropicalizing the characteristic varieties}
\label{sect:trop-cv}

Throughout this section, $X$ will be a space having the homotopy 
type of a connected CW-complex with finite $q$-skeleton, for some 
$q\ge 1$. Set $n=b_1(X)$; to avoid trivialities, we will assume 
that $n> 0$. 

\subsection{Tropicalizing affine varieties}
\label{subsec:trop-affine}
As in \S\ref{sect:trop}, we let  $\K=\C\bra{t}$ be the field of Puiseux series with 
complex coefficients. Recall this is an algebraically closed field which supports 
a non-Archimedean valuation, $\val\colon \K^{\times}\to \Q$.  
Letting $\val^{\times n}\colon (\K^{\times})^n\to \Q^n\subset\R^n$ 
be the $n$-fold product of this map, we define a ``valuation map" 
on the whole $\K$-character variety of $\pi_1(X)$, 
\begin{equation}
\label{eq:valmap}
\begin{tikzcd}[column sep=16pt]
\nu_X\colon H^1(X,\K^{\times})\ar[r]& \Q^n \subset \R^n ,
\end{tikzcd}
\end{equation}
by requiring that $\nu_X$ restricts to $\val^{\times n}$ on each connected 
component of $H^1(X,\K^{\times})=\coprod (\K^{\times})^n$.   In other words, 
if $\rho\colon \pi_1(X)\to \K^{\times}$ is a $\K$-valued multiplicative character, 
then the additive character $\val \circ \rho\colon \pi_1(X) \to \Q$ defines the 
element $\nu_X(\rho) \in H^1(X,\Q)=\Q^n\subset \R^n$.  

Alternatively, we may view the map $\nu_X$ as the composite%
\begin{equation}
\label{eq:valmap-bis}
\begin{tikzcd}[column sep=24pt]
H^1(X,\K^{\times}) \ar[r, "\val_*"]
& H^1(X,\Q) \ar[r] 
& H^1(X,\R)\, ,
\end{tikzcd}
\end{equation}
where the first arrow is the coefficient homomorphism 
induced by the map $\val\colon \K^{\times}\to \Q$, and the second arrow 
is the coefficient homomorphism $H^1(X,\Q)\inj H^1(X,\R)$. 
The map $\nu_X$ depends only on the abelianization of the group 
$G=\pi_1(X)$, so we may also denote it by $\nu_G$.  This map enjoys 
the following naturality property.  Let $\varphi\colon G\to K$ be a 
homomorphism between finitely generated groups.   We then 
have a commuting diagram,
\begin{equation}
\label{eq:homotopy}
\begin{gathered}
\begin{tikzcd}[column sep=3pc]
H^1(K,\K^{\times}) \ar[r, "\varphi^*"] \ar[d, "\nu_K"]
& H^1(G,\K^{\times})\phantom{\, .} \ar[d, "\nu_G"]\\
H^1(K,\R) \ar[r, "\varphi^*"] 
& H^1(G,\R)\, .
\end{tikzcd}
\end{gathered}
\end{equation}

\begin{definition}
\label{def:trop-w}
The {\em tropicalization}\/ of an algebraic subvariety 
$W\subset H^1(X,\C^{\times})$ is the closure in
 $H^1(X,\R)\cong \R^n$ of the image of 
 $W\times_{\C} \K \subset H^1(X,\K^{\times})$ 
under the map $\nu_X$,
\begin{equation}
\label{eq:trop-w}
\Trop(W)\coloneqq \overline{\nu_X (W\times_{\C} \K )}\, .
\end{equation}
\end{definition}

An extreme case is worth singling out.

\begin{lemma}
\label{lemma:trop-all}
If $W$ contains a connected component of $H^1(X,\C^{\times})$, 
then $\Trop(W)=H^1(X,\R)$.
\end{lemma}

\begin{proof}
Let $\T=H^1(X,\C^{\times})$, and let $\T^0=(\C^{\times})^n$ be the 
identity component. By assumption, $W$ contains $\rho \T^0$, 
for some $\rho\in \T$.  Therefore, $\Trop(W)$ contains 
$\Trop(\rho \T^0)=\Trop(\T^0)=\R^n$, and we're done.
\end{proof}

\subsection{Tropicalized characteristic varieties}
\label{subsec:trop-charvar}
Recall from \eqref{eq:cv-base} that 
$\VV^{i}(X, \K)=\VV^{i}(X)\times_{\C} \K$.  Applying Definition \ref{def:trop-w} 
to the characteristic varieties $\VV^{i}(X)$, where $i\le q$,  we have that 
\begin{equation}
\label{eq:trop-cv}
\Trop( \VV^{i}(X))=\overline{\nu_X\big(\VV^{i}(X,\K)\big)}\, .
\end{equation}
As noted in \eqref{eq:valmap-bis}, the map $\nu_X$ factors through the 
coefficient homomorphism $\val_*\colon H^1(X,\K^{\times}) \to H^1(X,\Q)$. 
By \eqref{eq:trop-rat}, the set of rational points on $\Trop( \VV^{i}(X))$ 
consists of all elements of the form $\nu_X(\rho)$, for some character 
$\rho \colon \pi_1(X)\to \K^{\times}$ which belongs to $\VV^{i}(X,\K)$; that is,
\begin{equation}
\label{eq:trop-cv-rat}
\Trop(\VV^{i}(X))\cap H^1(X,\Q) = \nu_X\big(\VV^{i}(X,\K)\big)\, .
\end{equation}

These tropical varieties are homotopty-type invariants.
Indeed, suppose $f\colon X\to Y$ is a homotopy equivalence. 
Then the induced isomorphism, $f^*\colon H^1(Y,\Z)\to H^1(X,\Z)$,  
defines both a monomial isomorphism $\alpha=f^*$ on $H^1(-,\C^{\times})$ 
and an $\R$-linear isomorphism $A=f^*$ on $H^1(-,\R)$. 
As we noted in \S\ref{subsec:prop-cv}, the map $\alpha$ takes 
$\VV^{1}(Y)$ to $\VV^{1}(X)$.  Using now formula \eqref{eq:trop-mat} 
in this slightly more general context, we conclude that 
$A(\Trop(\VV^{1}(Y)))=\Trop(\VV^{1}(X))$.

In degree $i=1$, the tropicalized characteristic varieties enjoy the 
following naturality property.

\begin{proposition}
\label{prop:trop-v1-nat}
Let  $G$ be a finitely generated group, and let $\varphi\colon G\surj Q$ be a 
surjective homomorphism.   Then the induced $\R$-linear map,   
$\varphi^*\colon H^1(Q,\R) \inj H^1(G,\R)$, restricts to an embedding 
$\Trop( \VV^{1}(Q)) \inj \Trop( \VV^{1}(G))$.
\end{proposition}

\begin{proof}
It follows from Proposition \ref{prop:v1-nat} that the induced morphism between 
$\K$-character groups, $\varphi^*\colon H^1(Q,\K^{\times}) \inj H^1(G,\K^{\times})$, 
restricts to an embedding $\VV^1(Q,\K) \inj \VV^1(G,\K)$.  Applying the valuation 
map from \eqref{eq:valmap} and using the commutativity of diagram \eqref{eq:homotopy}  
yields the claim. 
\end{proof}

The next result details the relationship between the exponential tangent cones 
and the tropicalizations of the characteristic varieties. 

\begin{proposition}
\label{prop:tau1 trop}
Let $X$ be a space as above. Then,
\begin{enumerate}
\item \label{tt1}
$\tau^{\R}_1(\VV^i(X))\subseteq \Trop(\VV^i(X))$, for all $i\le q$.
\item  \label{rt2}
Suppose there is a subtorus $T\subset \T_X^0$ such that 
$T \not\subset \VV^i(X)$, yet $\rho T\subset \VV^i(X)$ 
for some $\rho\in \T_X$. 
Then $\tau^{\R}_1(\VV^i(X))\subsetneqq \Trop(\VV^i(X))$.
\end{enumerate}
\end{proposition}

\begin{proof}
The first claim follows at once from Proposition \ref{prop:tc trop}, 
while the second claim follows from a slight modification of 
the proof of Proposition \ref{prop:t1-trop}.
\end{proof}

Finally, let  $G$ be a finitely generated group with $b_1(G)>0$, and let 
$H=G_{\ab}/\Tors(G_{\ab})$. Under some additional conditions on the 
Alexander polynomial $\Delta_G$, we can say more.

\begin{proposition}
\label{prop:trop-newt}
Suppose $\Delta_G$ is symmetric and 
$I^p_{H}\cdot ( \Delta_{G} ) = E_1(A_G)$, 
for some $p\ge 0$. Then the following hold for the tropical variety 
$Y=\Trop\big(\VV^1(G)\cap \T_G^0\big)$.
\begin{enumerate}
\item \label{lt1} $Y=\Trop(V(\Delta_G))\cup \{0\}$.
\item \label{lt2} $Y=-Y$. 
\item \label{lt3} $Y$ coincides with the positive-codimension skeleton of 
$\mathcal{F}(B_A)$, the face fan of the unit ball in the Alexander norm.
\end{enumerate}
\end{proposition}

\begin{proof}
Since the Alexander ideal $( \Delta_{G} )$ satisfies the assumption of 
Proposition \ref{prop:zz1-pi}, we infer that 
$\VV^1(G)\cap \T_G^0=V(\Delta_G)\cup \{1\}$; 
the first claim follows at once. 

Let $P=\Newt(\Delta_G)$ be the Newton polytope of $\Delta_G$ 
inside $H_1(G,\R)=H\otimes \R$. From \eqref{eq:normal-newt}, 
we know that $\Trop(V(\Delta_G))$ coincides with $\mathcal{N}(P)^{\codim >0}$, 
the positive-codimension skeleton of the inner normal fan to $P$. 
Now, since $\Delta_G$ is symmetric, we have that $P=-P$, 
and the second claim follows.  

Finally, we know from \eqref{eq:alex-ball} that $P$ is twice the polar 
dual of $B_A$.  Moreover, $0\in \operatorname{int}(P)$; thus, by \eqref{eq:polarity},  
the inner normal fan to $P$ is the face fan of $B_A$. The last claim follows.
\end{proof}

\section{Bieri--Neumann--Strebel--Renz invariants}
\label{sec:bnsr}

In this section, we review the definition of the Sigma-invariants 
of a group $G$ and, more generally, of a space $X$, following 
the approach from \cite{Bi07, FGS, PS-plms, Su-pisa12}.

\subsection{The \texorpdfstring{$\Sigma$}{Sigma}-invariants of a chain complex}
\label{subsec:bnsr cc}
Let $C=(C_i,\partial_i)_{i\ge 0}$ be a chain complex over a ring $R$, 
and let $q$ be a positive integer.  We say $C$ is of {\em finite $q$-type}\/ 
if there is a chain complex $C'$ of finitely generated, projective 
(left) $R$-modules and a chain map $C'\to C$ inducing isomorphisms 
$H_i(C')\to H_i(C)$ for $i<q$ and an epimorphism 
$H_q(C')\to H_q(C)$. For a free chain complex $C$, this is equivalent 
to being chain-homotopy equivalent to a free chain complex $D$ 
for which $D_i$ is finitely generated for all $i\leq q$.

Now let $G$ be a finitely generated group.  The character sphere 
\begin{equation}
\label{eq:sg}
S(G)\coloneqq (\Hom(G,\R)\setminus\{0\})/\R^{+},
\end{equation}
is the set of nonzero homomorphisms $G\to \R$ modulo 
homothety.  To simplify notation, we will usually denote both 
a nonzero homomorphism $\chi\colon G\to \R$ and its equivalence 
class, $[\chi] \in S(G)$, by the same symbol, $\chi$.  The character 
sphere may be identified with the unit sphere $S^{n-1}$ in the 
real vector space $\Hom(G,\R)\cong \R^n$, where $n=b_1(G)$.  

Given a nonzero homomorphism $\chi\colon G\to \R$,  the set  
$G_{\chi}\coloneqq\{ g \in G \mid \chi(g)\ge 0\}$ 
is a submonoid of $G$, which depends only on  
$[\chi]\in S(G)$.  The monoid ring 
$\Z{G}_{\chi}$ is a subring of the group ring $\Z{G}$; 
thus, any $\Z{G}$-module naturally acquires the structure of a 
$\Z{G}_{\chi}$-module, by restriction of scalars.

\begin{definition}[\cite{FGS}]  
\label{def:sigma chain}
Let $C$ be a chain complex over $\Z{G}$.  For each 
integer $q\ge 0$, the {\em $q$-th 
Bieri--Neumann--Strebel--Renz invariant}\/ of $C$ 
is the set
\begin{equation}
\label{eq:sigmakc}
\Sigma^q(C)=\big\{\chi\in S(G) \mid 
\text{$C$ is of finite $q$-type over $\Z{G}_{\chi}$}\big\}\, .
\end{equation}
\end{definition}

Suppose now that $N\triangleleft G$ is a normal 
subgroup such that the quotient group $G/N$ is abelian.  Let  
$S(G,N)$ be the set of homomorphisms $\chi \in S(G)$ 
for which $N\le \ker (\chi)$. Then $S(G,N)$ 
is a great subsphere of $S(G)$, obtained by intersecting 
the unit sphere $S(G)\subset H^1(G,\R)$ 
with the image of the linear map 
$\kappa^*\colon H^1(G/N,\R) \inj H^1(G,\R)$, 
where  $\kappa\colon G\surj G/N$ is the canonical projection. 
Finally, let $C$ be a chain complex of free $\Z{G}$-modules, with 
$C_i$ finitely generated for $i\le q$, and let $N$ be a 
normal subgroup of $G$ such that $G/N$ is abelian. 
Then, as shown in \cite{FGS},  $C$ is of finite $q$-type when 
restricted to $\Z{N}$ if and only if $S(G,N)\subset \Sigma^q(C)$.

\subsection{The \texorpdfstring{$\Sigma$}{Sigma}-invariants of a CW-complex}
\label{subsec:fgs cw}
Let $X$ be a connected CW-complex with 
finite $1$-skeleton, and let $G=\pi_1(X,x_0)$ be its fundamental 
group.  A choice of classifying map $X\to K(G,1)$ yields   
an induced isomorphism, $H^1(G,\R) \isom H^1(X,\R)$, 
which identifies the respective unit spheres, 
$S(G)$ and $S(X)$.  The cell structure on $X$ lifts to a 
cell structure on the universal cover $\wX$, invariant under the 
action of $G$ by deck transformations. 
Thus, the cellular chain complex $C_*(\wX,\Z)$ 
is a chain complex of free $\Z{G}$-modules.  

\begin{definition}
\label{def:fgs cw}
For each $q> 0$, the {\em $q$-th 
Bieri--Neumann--Strebel--Renz invariant}\/ 
of $X$ is the subset of $S(X)$ given by
$\Sigma^q(X,\Z)=\Sigma^q(C_*(\wX,\Z))$.
\end{definition}

We will denote by $\Sigma^q(X,\Z)^{\compl}$ the complement 
of $\Sigma^q(X,\Z)$ in $S(X)$.  It is shown in \cite{FGS} that 
$\Sigma^q(X,\Z)$ is an open subset of $S(X)$, which depends 
only on the homotopy type of $X$.  
The $\Sigma$-invariants enjoy the following naturality property.

\begin{lemma}[\cite{FGS}] 
\label{lem:fgs-nat}
Let $f \colon X \to Y$ be a map between two finite CW-complexes, 
and assume there is a map $g \colon Y \to X$ such that $f\circ g\simeq \id_Y$. 
If $f^*\colon H^1(Y,\R)\to H^1(X,\R)$ is the induced 
homomorphism and $f^*\colon S(Y) \to S(X)$ is its restriction to 
character spheres, then 
$(f^*)^{-1}\big(\Sigma^q(X,\Z)\big) \subseteq  \Sigma^q(Y,\Z)$, for all $q > 0$. 
\end{lemma}
Consequently, if $f$ is a homotopy equivalence, then 
$f^*(\Sigma^q(Y,\Z)) = \Sigma^q(X,\Z)$. 

We say that a subset $\Sigma\subset S^{n-1}$ is {\em symmetric}\/ if it 
is invariant under the antipodal map, i.e., $\Sigma=-\Sigma$. 
In general, the BNSR invariants are not symmetric; we will illustrate  
this phenomenon in Examples \ref{ex:bs12}, \ref{ex:deleq}, and \ref{ex:brown}. 
Nevertheless, as a consequence of Lemma \ref{lem:fgs-nat}, 
we have the following symmetry criterion.

\begin{proposition}
\label{prop:symm}
Suppose there is a homotopy equivalence 
$f\colon X\to X$ such that $f_*\colon H_1(X, \R) \to H_1(X, \R)$ 
is equal to $-\id_{H_1(X,\R)}$. 
Then $\Sigma^q(X,\Z)= -\Sigma^q(X,\Z)$, for all $q\ge 1$.
\end{proposition}

\subsection{The \texorpdfstring{$\Sigma$}{Sigma}-invariants of a group}
\label{subsec:sigma-groups} 

Let $G$ be a finitely generated group, and let $\Cay(G)$ be the Cayley 
graph associated to a fixed finite generating set. The invariant $\Sigma^1(G)$ 
of Bieri, Neumann and Strebel \cite{BNS} is the set of homomorphisms 
$\chi\in S(G)$ for which the induced subgraph of $\Cay(G)$ on vertex 
set $G_{\chi}$  is connected. The BNS set is an open subset of $S(G)$, 
which does not depend on the choice of finite generating set for $G$. 
The rational points on $S(G)$ correspond to epimorphisms 
$\chi\colon G\to \Z$; the kernel of $\chi$ is 
finitely generated if and only if both $\chi$ and $-\chi$ belong to $\Sigma^1(G)$. 
Additionally, the complements of the $\Sigma$-invariants enjoy 
the following naturality property.  

\begin{proposition}[\cite{BNS}] 
\label{prop:sigma1-nat}
Suppose $\varphi\colon G\surj Q$ is a 
surjective group homomorphism. Then the induced embedding, 
$\varphi^*\colon S(Q)\inj S(G)$, restricts to an injective map 
between the complements of the respective BNS-invariants,
$\varphi^*\colon \Sigma^1(Q)^{\compl} \inj  \Sigma^1(G)^{\compl}$.
\end{proposition}

The BNS invariant was generalized by Bieri and Renz \cite{BR}, 
as follows. For a $\Z{G}$-module $M$, define  
$\Sigma^{q}(G, M)\coloneqq\Sigma^{q}(F_{\bullet})$, where 
$F_{\bullet} \to M$ is a projective $\Z{G}$-resolution of $M$.  
In particular, this yields invariants $\Sigma^{q}(G,\Z)$, 
where $\Z$ is viewed as a trivial $\Z{G}$-module; 
clearly, $\Sigma^{q}(G,\Z)=\Sigma^{q} (K(G,1),\Z)$.  
Likewise, we have the invariants $\Sigma^{q}(G,\k)$, 
where $\k$ is a field. There is always an inclusion 
$\Sigma^{q}(G,\Z) \subseteq \Sigma^{q}(G,\k)$, 
but this inclusion may be strict.  

As noted in \cite[\S 1.3]{BR}, $\Sigma^1(G)=-\Sigma^1(G,\Z)$.  
It follows from Proposition \ref{prop:symm} that $\Sigma^1(G)$ 
is symmetric whenever $G$ admits an automorphism inducing 
minus the identity on $G_{\ab}\otimes \R$. 

\subsection{Novikov--Sikorav completion}
\label{subsec:novikov}
Let $G$ be a group. 
The {\em Novikov--Sikorav completion}\/ of the group ring $\Z{G}$ 
with respect to a homomorphism $\chi\colon G\to \R$ consists of 
all formal sums $\sum n_g g\in \Z^{G}$, 
having the property that, for each $c\in \R$, the set 
\begin{equation}
\label{eq:nov-sik}
\big\{g\in G \mid \text{$n_g \ne 0$ and $\chi(g) \ge c$}\big\}
\end{equation} 
is finite, see \cite{Novikov, Sk87,Far}. 
With the usual addition and with multiplication defined by 
$(\sum n_g g) \cdot  (\sum m_h h) = \sum (n_g m_h) g h$, 
the Novikov--Sikorav completion, $\widehat{\Z{G}}_{\chi}$, is a ring 
containing $\Z{G}$ as a subring. Consequently, $\widehat{\Z{G}}_{\chi}$ 
carries a natural structure of left $\Z{G}$-module (we will also view it as a right 
$\widehat{\Z{G}}_{\chi}$-module). For instance, if 
$G=\Z=\langle t \rangle$ and $\chi(t)=1$, then 
$\widehat{\Z{G}}_{\chi}=\{\sum_{i\le k} n_i t^i \mid \text{$n_i\in \Z$, 
for some $k\in \Z$}\}$. 

To see why $\widehat{\Z{G}}_{\chi}$ is a ring completion, let $U_m$ be 
the additive subgroup of $\Z{G}$ (freely) generated by the set 
$\{g\in G \mid \chi(g)\ge m\}$.  
Requiring the decreasing filtration $\{U_m\}_{m\in \Z}$ to form a 
basis of open neighborhoods of $0$ defines a topology 
on $\Z{G}$, compatible with the ring structure.  Then,
as noted in \cite[\S 4.2]{Bi07}, the Novikov--Sikorav ring is the 
completion of $\Z{G}$ with respect to this filtration:
\begin{equation}
\label{eq:completion}
\widehat{\Z{G}}_{-\chi}=\varprojlim_{m}\, \Z{G}/U_m\, .
\end{equation}

Moreover, the Novikov--Sikorav completion enjoys the following functoriality 
property.  Let $\phi\colon G\to K$ be a homomorphism, and 
let $\bar\phi\colon \Z^{G}\to \Z^{K}$ be its linear extension to 
formal sums. If $\chi\colon K\to \R$ is a character, then 
$\bar\phi$ restricts to a morphism of topological rings 
between the corresponding completions, $\hat\phi \colon  
\widehat{\Z{G}}_{\chi\circ \phi} \to  \widehat{\Z{K}}_{\chi}$. 

\subsection{Novikov--Sikorav homology}
\label{subsec:novikov-sikorav}
In his thesis \cite{Sk87}, J.-Cl.~Sikorav reinterpreted the BNS invariant 
of a finitely generated group $G$ in terms of Novikov homology. 
This interpretation was extended to all BNSR invariants 
by Bieri \cite{Bi07}, and later to the BNSR invariants of 
CW-complexes by Farber, Geoghegan and Sch\"{u}tz \cite{FGS}. 

Let $X$ be a connected CW-complex and let $G=\pi_1(X)$. 
Recall that the homology groups of $X$ with coefficients in a left 
$\Z{G}$-module $M$ are given by 
$H_{i}(X,M)\coloneqq H_i( C_*(\widetilde{X},\Z)\otimes_{\Z{G}} M )$, 
where $\widetilde{X}$ is the universal cover of $X$ and $C_*(\widetilde{X},\Z)$ 
is its cellular chain complex, viewed as a right $\Z{G}$-module via the 
cellular action of $G$ on the chains of $\widetilde{X}$. 
We will use in an essential way the following theorem, which expresses the 
BNSR invariants of $X$ as the vanishing loci for homology with 
coefficients in the Novikov--Sikorav completions of $G$.  

\begin{theorem}[\cite{FGS}]
\label{thm:bns novikov}
If $X$ is a connected CW-complex with finite $q$-skeleton, then    
\begin{equation}
\label{eq:bnsr fgs}
\Sigma^{q}(X, \Z)= \big\{ \chi \in S(X) \mid
H_{i}(X, \widehat{\Z{G}}_{-\chi})=0,\: \text{for all $i\le q$} \big\}\, .
\end{equation}
\end{theorem}

In particular, the BNS set $\Sigma^1(G)=-\Sigma^1(G,\Z)$ consists 
of those characters $\chi \in S(G)$ for which both $H_{0}(G, \widehat{\Z{G}}_{\chi})$ 
and $H_{1}(G, \widehat{\Z{G}}_{\chi})$ vanish.

\begin{remark}
\label{rem:signs}
There were several steps along the way where we had to make a choice 
of sign in the various definitions. As much as possible, we used the original 
sign conventions, but the literature varies in many of the particulars, so much 
care must be taken; see for instance the discussions 
in \cite[\S 1.3]{BR}, \cite[\S 3]{FGS}, and \cite[\S 4.4]{FT20}.  
The sign discrepancy between the BNS invariant 
$\Sigma^1(G)$ and the BNSR invariant $\Sigma^1(G, \Z)$ is as noted by 
Bieri and Renz in \cite{BR}.  The inequality from \eqref{eq:nov-sik} used in 
the definition of the Novikov--Sikorav completion is 
as in \cite{Novikov, Far, FGS,PS-plms,FT20}, but the 
opposite of \cite{Bi07,Ki19}, whence the minus signs 
in \eqref {eq:completion} and \eqref{eq:bnsr fgs}. 
In the end, as done in \cite{FT20}, the best way to check that 
all those signs and inequalities are consistent is to verify the 
computations on the Baumslag--Solitar group $\operatorname{BS}_{1,2}$ 
(the simplest group $G$ for which $\Sigma^1(G)$ is not symmetric); 
we will also do this in Example \ref{ex:bs12}.
\end{remark}

\section{\texorpdfstring{$\Sigma$}{Sigma}-invariants, cohomology jump loci, and tropicalization}
\label{sect:sig-trop-char}

In this section we prove our main result, relating the BNSR invariants 
of a space $X$ to the tropicalization of its characteristic varieties. As before, 
we will assume that $X$ has the homotopy type of a connected CW-complex 
with finite $q$-skeleton, for some $q\ge 1$, and $b_1(X)> 0$.

\subsection{\texorpdfstring{$\Sigma$}{Sigma}-invariants and exponential tangent cones}
\label{subsec:bnscv}
 
We start by recalling the main result from \cite{PS-plms}, which establishes 
a bridge between the $\Sigma$-invariants of a space $X$ as above and 
the real points on the exponential tangent cones to the respective 
characteristic varieties.  As in the previous section, $S(V)$ denotes 
the intersection of the unit sphere $S(X)\subset H^1(X,\R)$ with a subset 
$V\subset H^1(X,\R)$; in particular, if $V=\{0\}$, then $S(V)=\myempty$. 

\begin{theorem}[\cite{PS-plms}]
\label{thm:bns tau}
Let $X$ be as above. Then, 
\begin{equation}
\label{eq:bns-tau bound}
\Sigma^q(X, \Z)\subseteq 
S( \tau^{\R}_1(\VV^{\le q}(X)))^{\compl}. 
\end{equation}
\end{theorem}

Qualitatively, this theorem says that each BNSR set
$\Sigma^{q}(X, \Z)$ is contained in the complement of a union of 
rationally defined great subspheres.  The proof of this result, given 
in \cite[Prop. 8.5 and Thm. 9.1]{PS-plms}, 
makes use of the Novikov Betti numbers, $b_i(X,\chi)$, associated 
to an additive character $\chi\in S(X)$, and involves  showing that
\begin{equation}
\label{eq:nov-tau}
-\chi\in \Sigma^q(X,\Z) \Rightarrow 
b_{\le q}(X, \chi)=0 \Leftrightarrow 
\chi\not\in \tau^{\R}_1\big(\VV^{\le q}(X)\big)\, .
\end{equation}
We will give in Corollary \ref{cor:sigma-trop-tau} a different proof 
of Theorem \ref{thm:bns tau}, which does not rely on the (forward) 
implications from \eqref{eq:nov-tau}, yet reaches a stronger conclusion. 

Now suppose the space $X$ is {\em  formal}, in the sense of rational homotopy 
theory.  Then, by the Tangent Cone Theorem from \cite{DP-ccm, DPS-duke},
\begin{equation}
\label{eq:tcone}
\tau_1(\VV^i(X))=\RR^i(X)
\end{equation}
for all $i\le q$, where $\RR^i(X)\subseteq H^1(X,\C)$ 
are the {\em resonance varieties}\/ associated to the cohomology algebra $H^*(X,\C)$; 
see \cite{Su-tcone, Su-tc3d} for more on this.  Letting 
$\RR^i(X, \R)=\RR^i(X)\cap H^1(X,\R)$ be the real resonance varieties, 
this allows us to replace  $\tau^{\R}_1(\VV^i(X))$ by $\RR^i(X,\R)$ in 
Proposition \ref{prop:tau1 trop}; moreover, the following inclusion 
holds in the formal setting,
\begin{equation}
\label{eq:bns-res bound}
\Sigma^q(X, \Z)\subseteq 
S\big( \RR^{\le q}(X,\R)\big)^{\compl}. 
\end{equation}

In some instances, which are treated in detail in  \cite{PS-plms}, the inclusions 
\eqref{eq:bns-tau bound} or \eqref{eq:bns-res bound} hold as equalities. 
We briefly recall those examples. 

\begin{example}
\label{ex:nilmanifold} 
Let $X$ be a nilmanifold.  Then $\Sigma^i(X, \Z) = S(X)$, 
while $\VV^i(X)=\{1\}$ and so $\tau^{\R}_1(\VV^{i}(X))=\{0\}$, for all $i$.  
Thus, $\Sigma^q(X, \Z)=S( \tau^{\R}_1(\VV^{\le q}(X)))^{\compl}$, for all $q$.
\end{example}

\begin{example}
\label{ex:RAAG} 
Associated to every finite simple graph $\Gamma$ there is a 
right-angled Artin group, $G=G_{\Gamma}$. Such a group admits 
as classifying space a finite CW-complex which is formal. As shown 
in \cite{PS-plms} (based on computations from \cite{MMV} and \cite{PS-adv}), 
the  equality $\Sigma^q(G,\R) = S(\RR^{\le q}(G,\R))^{\compl}$ holds for all $q$.  
Moreover, $\Sigma^q(G, \Z)= S(\RR^{\le q}(G, \R))^{\compl}$, provided  
the homology groups of certain subcomplexes in the 
flag complex of $\Gamma$ are torsion-free. This condition is always 
satisfied in degree $q=1$, giving $\Sigma^1(G) =  S(\RR^1(G, \R))^{\compl}$.
\end{example}

In general, though---as we shall see in a number of examples in the last 
few sections---the inclusions \eqref{eq:bns-tau bound} 
and \eqref{eq:bns-res bound} are strict, even when $q=1$.

\subsection{Characters and valuations}
\label{subsec:char-val}

In order to establish our main result, we will rely instead on \cite[Thm.~10.1]{PS-plms}.
For completeness, we outline the proof of this theorem, with some additional 
details and explanations provided.

\begin{theorem}[\cite{PS-plms}]
\label{thm:val}
Let $X$ be a space as above, and let $\k$ be an arbitrary field. Suppose 
$\rho\colon \pi_1(X)\to \k^{\times}$ is a multiplicative character 
such that $\rho \in \VV^{\le q}(X, \k)$. 
Let $\upsilon\colon \k^{\times}\to \R$ be the homomorphism 
defined by a valuation on $\k$, and write $\chi=\upsilon\circ \rho$.  
If the additive character $\chi\colon \pi_1(X)\to \R$ 
is nonzero, then $\chi \not\in \Sigma^q(X, \Z)$.
\end{theorem}

\begin{proof}
Let $\hat{\k}$ be the topological completion of $\k$ with respect to the absolute  
value defined by the valuation $\upsilon$.  Then $\hat{\k}$ is a field, and 
the map to the completion, $\iota\colon \k \inj \hat{\k}$, is a field extension.

Now let $G=\pi_1(X)$. Every character $\rho\colon G\to \k^{\times}$ extends linearly to 
a ring map, $\bar\rho\colon \Z{G}\to \k$.  Since $\chi=\upsilon\circ \rho$, 
formula \eqref{eq:completion} allows us to extend 
$\bar\rho$ to a morphism of topological rings, 
$\hat{\rho}\colon \widehat{\Z{G}}_{-\chi}\to \hat{\k}$.  
This makes $\hat{\k}$ into a $\widehat{\Z{G}}_{-\chi}$-module, 
denoted $\hat{\k}_{\hat{\rho}}$; restricting scalars via the inclusion  
$ \Z{G} \inj \widehat{\Z{G}}_{-\chi}$ yields the $\Z{G}$-module 
$\hat{\k}_{\iota\circ \rho}$, defined by 
the character $\iota\circ \rho\colon G\to \hat\k^{\times}$.

For a ring $R$, a bounded below chain complex of flat right $R$-modules 
$K_{*}$, and a left $R$-module $M$, there is a (right half-plane, boundedly 
converging) K\"{u}nneth spectral sequence, 
\begin{equation}
\label{eq:base-change}
E^2_{ij}=\Tor^R_i(H_j(K),M) \Rightarrow H_{i+j}(K\otimes_R M)\, ,
\end{equation}
see \cite[Thm.~5.6.4]{Weibel}. 
We will apply this spectral sequence to the ring $R=\widehat{\Z{G}}_{-\chi}$, 
the chain complex of free $R$-modules 
$K_*= C_*\big(\widetilde{X},\Z\big) \otimes_{\Z{G}} \widehat{\Z{G}}_{-\chi}$, and the 
$R$-module $M=\hat{\k}_{\hat\rho}$. 

Let $\rho \in \VV^{\le q}(X, \k)$, and suppose that 
$\chi=\upsilon\circ \rho$ belongs to $\Sigma^q(X, \Z)$. 
By Theorem \ref{thm:bns novikov}, this condition is 
equivalent to the vanishing of $H_{j}(X, \widehat{\Z{G}}_{-\chi})$ for all $j\le q$;   
that is, $H_{j}(K) =0$ for $j\le q$.  Therefore, $E^2_{ij}=0$ for $j\le q$. 
Noting that 
\begin{equation}
\label{eq:coefficients}
K\otimes_R M= C_*\big(\widetilde{X},\Z\big) \otimes_{\Z{G}} 
\widehat{\Z{G}}_{-\chi} \otimes_{\widehat{\Z{G}}_{-\chi}} \hat{\k}_{\hat\rho} 
=  C_*\big(\widetilde{X},\Z\big) \otimes_{\Z{G}} \hat\k_{\iota\circ \rho}\, ,
\end{equation}
we infer from \eqref{eq:base-change}  that 
$H_{i+j}(X,\hat{\k}_{\iota\circ \rho})=0$ for $j\le q$, and so 
$H_{j}(X,\hat{\k}_{\iota\circ \rho})=0$ for $j\le q$. From the definition 
of the characteristic varieties, this is equivalent to 
$\iota\circ \rho \notin \VV^{\le q}(X, \hat\k)$. Hence, by \eqref{eq:cv-ext}, 
$\rho \notin \VV^{\le q}(X, \k)$, contradicting our hypothesis on $\rho$. 
Therefore, $\chi\notin \Sigma^q(X, \Z)$, and we are done.
\end{proof}

The above theorem builds on an idea that goes back to Bieri and Groves \cite{BGr}.
A particular case of Thorem \ref{thm:val}---for $X=K(G,1)$, $q=1$, and $\upsilon$ a 
discrete valuation---was previously proved by Delzant in \cite[Prop.~1]{De08}, 
using the interpretation of $\Sigma^1(G)$ in terms of $G$-actions on trees 
given by Brown in \cite{Br}.

\subsection{Main result}
\label{subsec:main}

We are now in a position to state and prove our main result. 
Once again, let $X$ be a connected CW-complex 
with finite $q$-skeleton, for some $q\ge 1$.  We place ourselves 
in the framework from Sections \ref{sect:trop} and \ref{sect:trop-cv}, and work 
with the base field $\k=\C$ and its extension to the field of Puiseux series, 
$\K=\C\bra{t}$, endowed with its usual valuation map, $\val\colon \K^{\times}\to \Q$. 
 
\begin{theorem}
\label{thm:bns-trop}
Let $\VV^{\le q}(X)\subset H^1(X,\K^{\times})$ be the union of the characteristic 
varieties of $X$ in degrees up to $q$, and let $\Trop( \VV^{\le q}(X))\subset H^1(X,\R)$ 
be its tropicalization. Then 
\begin{equation}
\label{eq:bound}
\Sigma^q(X,\Z) \subseteq S\big(\!\Trop( \VV^{\le q}(X))\big)^{\compl}.
\end{equation}
\end{theorem}

\begin{proof}
Let $\rho\in H^1(X,\K^{\times})$ be a $\K$-valued, multiplicative character.  
Composing the valuation map 
$\val\colon \K^{\times}\to \Q$ with the homomorphism 
$\rho\colon \pi_1(X)\to \K^{\times}$, we obtain  
an additive character, $\chi\coloneqq\val\circ \rho\colon \pi_1(X) \to \Q$, 
which may be viewed as an element of $H^1(X,\Q)$, that is, a rational 
point on $H^1(X,\R)$.  Moreover, if $\chi$ is nonzero, then $\chi$ 
determines a rational point on $S(X)$.

Now suppose $\rho$ belongs to the characteristic variety 
$\VV^{\le q}(X, \K)=\VV^{\le q}(X) \times_{\C} \K$. 
By \eqref{eq:trop-cv-rat}, the set of rational points on the 
tropicalization of $\VV^{\le q}(X)$ is $\nu_X\big(\VV^{\le q}(X, \K)\big)$.
Thus, $\chi=\val\circ \rho=\nu_X(\rho)$ is a rational point on 
$\Trop(\VV^{\le q}(X))=\overline{\VV^{\le q}(X, \K)}$, and 
conversely, all rational points on 
$\Trop(\VV^{\le q}(X))$ are of the form $\nu_X(\rho)$, for 
some $\rho\in \VV^{\le q}(X, \K)$.

Finally, assume that  $\chi$ is also nonzero, so that it 
represents an (arbitrary) rational point in $S(\Trop( \VV^{\le q}(X))$. 
Then, by Theorem \ref{thm:val}, the additive character $\chi$ 
belongs to $\Sigma^q(X,\Z)^{\compl}$.  
By the above, though, the set of rational points is dense in 
$S(\Trop(  \VV^{\le q}(X)))$. Moreover, we also know 
from \S\ref{subsec:fgs cw} that 
$\Sigma^q(X,\Z)^{\compl}$ is a closed subset of $S(X)$.  Thus, 
$S(\Trop( \VV^{\le q}(X)))\subseteq \Sigma^q(X,\Z)^{\compl}$. 
\end{proof}

Since we are working over the field $\k=\C$, we also have at our disposal 
the exponential map, $\exp\colon \C\to \C^{\times}$, and the resulting exponential 
tangent cone construction, $\tau_1(W)$, for subvarieties $W\subset (\C^{\times})^n$. 
The next result recovers (in a much stronger form) Theorem \ref{thm:bns tau}. 

\begin{corollary}
\label{cor:sigma-trop-tau}
With notation as above, 
\begin{equation}
\label{eq:2-inclusions-spaces}
\Sigma^q(X,\Z) \subseteq S(\Trop( \VV^{\le q}(X)))^{\compl}
\subseteq S( \tau^{\R}_1(\VV^{\le q}(X)))^{\compl}. 
\end{equation}
\end{corollary}
\begin{proof}
The first inclusion comes from Theorem \ref{thm:bns-trop}, while the second 
inclusion follows from Proposition \ref{prop:tc trop}.
\end{proof}  

In a very special situation (which we will encounter several times later on), 
the above theorem allows us to precisely identify the sets $\Sigma^q(X,\Z)$, 
even in cases when Theorem \ref{thm:bns tau} does not.

\begin{corollary}
\label{cor:sigmaq-X}
If $\VV^{\le q}(X)$ contains a connected component 
of $H^1(X,\C^{\times})$, then $\Sigma^q(X,\Z)=\myempty$.
\end{corollary}

\begin{proof}
By Lemma \ref{lemma:trop-all}, $\Trop (\VV^{\le q}(X))=H^1(X,\R)$. 
The claims now follows from Theorem \ref{thm:bns-trop}.
\end{proof}

\section{Upper bounds for the BNS invariant}
\label{sect:bns-trop-bound}

In most of the applications and examples described in the 
following sections, we will only deal with the original BNS sets 
$\Sigma^1(X)=\Sigma^1(\pi_1(X))$, which recall are equal to 
$-\Sigma^1(X,\Z)$.  For completeness, we restate our results 
from \S\ref{subsec:main} in this setting. 

\begin{corollary}
\label{cor:sigma1-X}
Let $X$ be a connected CW-complex with finite $1$-skeleton. Then
\begin{equation}
\label{eq:sigma1-X-bound}
\Sigma^1(X) \subseteq -S(\Trop (\VV^1(X)))^{\compl}.
\end{equation}
\end{corollary}

\begin{corollary}
\label{cor:bns-bounds-group}
Let $G$ be a finitely generated group. Then
\begin{equation}
\label{eq:two inclusions-G}
\Sigma^1(G) \subseteq -S(\Trop( \VV^{1}(G)))^{\compl}
\subseteq S( \tau^{\R}_1(\VV^{1}(G)))^{\compl}.
\end{equation}
\end{corollary}

\begin{corollary}
\label{cor:sigma-empty-2}
If $\VV^1(G)$ contains a connected component of $\T_{G}$, then 
$\Sigma^1(G)=\myempty$.
\end{corollary}

We shall see applications of the last corollary in 
Sections \ref{sect:orbifolds}--\ref{sect:seifert} on 
orbifold Riemann surfaces and Seifert manifolds.

Under some more restrictive hypothesis, it is possible to recast these 
tropical upper bounds in terms of simpler, more computable data. 
Given a finitely generated group $G$ with $b_1(G)>0$, we prove 
two such results: one in which the upper bound for $\Sigma^1(G)$ 
is given in terms of the tropicalization of the Alexander polynomial 
$\Delta_G\in \Z{H}$, where $H=G_{\ab}/\Tors(G_{\ab})$,  and the other in 
which the upper bound is given in purely homological terms, using a 
family of  projections onto suitably chosen quotients of $G$.
 
\begin{theorem}
\label{thm:sigma1-G}
Suppose $\Delta_G$ is symmetric and 
$I^p_{H}\cdot ( \Delta_{G} ) \subseteq E_1(A_G)$, for some $p\ge 0$. 
Then
\begin{equation}
\label{eq:sigma-facets}
\Sigma^1(G) \subseteq \bigcup_F S(F)\, ,
\end{equation} 
where $F$ runs through the open facets of $B_A$.
\end{theorem}

\begin{proof}
By Corollary \ref{cor:bns-bounds-group}, 
$\Sigma^1(G)$ is contained in $-S(\Trop( \VV^{1}(G)))^{\compl}$;  
in turn, this set is contained in $-S(\Trop( \VV^{1}(G)\cap \T_G^0))^{\compl}$.  

By Proposition \ref{prop:trop-newt}, the set 
$\Trop\big(\VV^1(G)\cap \T_G^0\big)$ is symmetric  and 
coincides with the positive-codimension skeleton of $\mathcal{F}(B_A)$, 
the face fan of the unit ball in the Alexander norm.  Thus, its complement 
in $H^1(G,\R)$ is the union of the open cones on the facets of $B_A$. 
The claim follows.
\end{proof}

We shall give applications of this theorem in Sections \ref{sect:one-rel}--\ref{sect:3mfd} 
on $1$-relator groups and $3$-manifold groups.

\begin{theorem}
\label{thm:sigma-pencils}
Let $G$ be a finitely generated group, and let $f_\alpha\colon G\to G_\alpha$ 
be a finite collection of epimorphisms.  Suppose each characteristic variety 
$\VV^1(G_\alpha)$ contains a component of $\T_{G_\alpha}$.  Then 
\begin{equation}
\label{eq:sigma-pen}
\Sigma^1(G) \subseteq \bigg(\bigcup_{\alpha} 
S\big(f^*_{\alpha}(H^1(G_\alpha,\R))\big)\bigg)^{\compl}.
\end{equation}
\end{theorem}

\begin{proof}
Fix an index $\alpha$, and consider the induced homomorphism
$f^*_{\alpha}\colon H^1(G_{\alpha},\R)\to H^1(G,\R)$.  Since 
$f_{\alpha}$ is assumed to be surjective, $f^*_{\alpha}$ is 
an $\R$-linear injection, and, by Proposition \ref{prop:trop-v1-nat}, 
the map $f^*_{\alpha}$  takes $\Trop(\VV^1(G_\alpha))$ to $\Trop(\VV^1(G))$. 
Therefore, 
\begin{equation}
\label{eq:trop-cup}
\medmath{\bigcup_{\alpha}}\, f_{\alpha}^*\big(\Trop(\VV^1(G_\alpha))\big)\subseteq 
\Trop( \VV^{1}(G)) \, .
\end{equation}

By our other assumption, we have that $\VV^1(G_\alpha)\supseteq \rho \T^0_{G_{\alpha}}$, 
for some (torsion) character $\rho_\alpha\in \T_{G_\alpha}$. Thus, by Lemma \ref{lemma:trop-all}, 
$\Trop\big(\VV^1(G_\alpha)\big)=H^1(G_{\alpha},\R)$. Therefore,
\begin{align}
\label{eq:sig-siga}
\Sigma^1(G) 
&\subseteq -S(\Trop( \VV^{1}(G)))^{\compl} 
&&\text{by Corollary \ref{cor:bns-bounds-group}}
\notag\\
&\subseteq -\Big(\medmath{\bigcup_{\alpha}}\, 
S\big(f_{\alpha}^*(\Trop(\VV^1(G_\alpha))\big)\Big)^{\compl}
&&\text{by \eqref{eq:trop-cup}}\\
&= \Big(\medmath{\bigcup_{\alpha}}\,
S\big(f_{\alpha}^*(H^1(G_\alpha,\R))\big)\Big)^{\compl}, \notag
\end{align}
where at the last step we used the aforementioned equality, together with the homogeneity 
of the linear subspace $f_{\alpha}^*(H^1(G_\alpha,\R))$ to dispense with the sign. 
\end{proof}

\begin{remark}
\label{rem:alt-proof}
It is possible to give an alternative argument, relying on 
Proposition \ref{prop:sigma1-nat} instead of Proposition \ref{prop:trop-v1-nat}.  
In this approach, one starts by noting that each map $f^*_\alpha\colon S(G_\alpha) \inj S(G)$ 
is an embedding taking $\Sigma^1(G_{\alpha})^{\compl}$ to $\Sigma^1(G)^{\compl}$, 
and then proceeds in like manner.
\end{remark}

We shall provide applications of this theorem in 
Sections \ref{sect:kahler}--\ref{sect:arrs} on K\"{a}hler manifolds and 
hyperplane arrangements.

\section{One-relator groups}
\label{sect:one-rel} 

In \cite{Br}, Brown gave an algorithm for computing the 
BNS invariant of a $1$-relator group $G$  (see \cite{BR, FT20} 
for other approaches). If  $G$ has at least $3$ generators, then 
$\Sigma^1(G)=\myempty$, so the interesting case is that of a 
$2$-generator, $1$-relator group, $G=\langle x_1,x_2\mid r\rangle$.  

First suppose $b_1(G)=1$. 
Let $\chi\colon G\to \R$ be a nonzero homomorphism, 
and assume without loss of generality that $\chi(x_1)>0$.  
Then one determines whether $[\chi]$ belongs to $\Sigma^1(G)$ 
according to whether the ``leading term" of $\partial_2(r)$ in 
the $\chi$-direction is of the form $\pm g\in \Z{G}$. 

\begin{example}
\label{ex:z*z2}
Let $G=\Z*\Z_2=\langle x_1,x_2\mid x_2^2\rangle$. 
Then  $S(G)=\{\pm \chi\}$, where $\chi(x_1)=1$ 
and $\chi(x_2)=0$. The Fox derivative $\partial_2(r)=1+x_2$ 
has leading term $1+x_2$ in both directions; thus, $\Sigma^1(G)=\emptyset$. 
On the other hand, $\T_G=\C^{\times}\times \{\pm 1\}$ and 
$\VV^1(G)=\{1\}\cup \C^{\times} \times \{-1\}$.  Hence, 
$\Sigma^1(G)=S(\Trop (\VV^1(G)))^{\compl}$, although 
$S(\tau^{\R}_1(\VV^1(G)))^{\compl}=\{\pm \chi\}$.
\end{example}

\begin{example}
\label{ex:bs12}
Let $G=\langle x_1,x_2\mid x_1^{\,}x_2x_1^{-1}x_2^{-2}\rangle$ 
be the Baumslag--Solitar group $\operatorname{BS}_{1,2}$.
Then  $S(G)=\{\pm \chi\}$, where $\chi(x_1)=1$ 
and $\chi(x_2)=0$. The Fox derivative $\partial_2(r)=x_1-x_2-1$ 
has leading term $x_1$ in the $\chi$-direction, and $-(1+x_2)$ in the 
$-\chi$-direction; thus, $\Sigma^1(G)=\{\chi\}$. On the other hand, 
$\T_G=\C^{\times}$ and $\VV^1(G)=\{1,2\}$; 
thus, $\Trop (\VV^1(G))=\tau^{\R}_1(\VV^1(G))=\{0\}$, and so 
inclusion \eqref{eq:sigma1 bound} is strict in this case.
\end{example}

Now suppose $b_1(G)=2$.  Then, in fact, $G_{\ab}\cong \Z^2$, 
and without loss of essential generality, we may assume 
that $r$ is a nontrivial, cyclically reduced word in the 
commutator subgroup of the free group on generators 
$x_1$ and $x_2$. Let $P$ be the boundary 
of the convex hull of the walk traced out by the relator $r$ in the 
$x_1$--$x_2$ plane $\R^2$; vertices of the polygon $P$ 
traversed only once are called simple, while horizontal or 
vertical edges are called special if they contain precisely 
two simple vertices. Then  $\Sigma^1(G)$ is a finite union of open 
arcs on the unit circle $S^1$: for each simple vertex $v$ or 
special edge $e$ of $P$, the corresponding arc runs from 
the normal vector to the edge preceding $v$ or $e$ to the 
one following it (when proceeding in a counterclockwise direction). 

On the other hand, the Alexander module $A_G$ is presented by 
the matrix $\big( \partial_1(r)^{\ab}\ \partial_2(r)^{\ab}\big)$, while 
the Alexander polynomial $\Delta_G$ is the gcd of the 
entries of this matrix; hence, 
$I_{G_{\ab}}\cdot ( \Delta_{G} ) = E_1(A_G)$. 
By Proposition \ref{prop:zz1-pi}, we have $\VV^1(G)=V(\Delta_G)\cup \{1\}$, 
and so the tropical upper bound from \eqref{eq:two inclusions-G} 
takes the form 
\begin{equation}
\label{eq:sigma1 bound}
\Sigma^1(G) \subseteq -S(\Trop(V(\Delta_G)))^{\compl},
\end{equation}
where $\Trop(V(\Delta_G))$ is the union of the rays in 
$H^1(G,\R)=\R^2$ starting at $0$ and inner 
normal to the edges of the Newton polygon of $\Delta_G$.  
On the other hand, $\tau^{\R}_1(\VV^1(G))$ is a union of lines 
through $0$ with rational slopes, and equals $\{0\}$ if $\Delta_G(1)\ne 0$. 
We illustrate the way this works in three examples 
(corresponding to the first three pictures from Figure \ref{fig:bns}).

\begin{figure}
\centering
\begin{tikzpicture}[baseline=(current bounding box.center),scale=0.47]  
\tikzset{mypoints/.style={fill=white,draw=black,thick}}
\def\ptsize{2.5pt}
\clip (0,0) circle (3);
\draw[name path=circle,red,ultra thick] (0,0) circle (2);
\draw[ style=thick, color=blue ] (-3,0) -- (3,0);
\fill[mypoints] (-2,0) circle (\ptsize);
\fill[mypoints] (2,0) circle (\ptsize);
\end{tikzpicture}
\hspace{1pc}
\begin{tikzpicture}[baseline=(current bounding box.center),scale=0.47]  
\tikzset{mypoints/.style={fill=white,draw=black,thick}}
\def\ptsize{2.5pt}
\clip (0,0) circle (3);
\draw[name path=circle,red,ultra thick] (0,0) circle (2);
\draw[ style=thick, color=blue ] (0,-3) -- (0,0);
\draw[ style=thick, color=blue ] (-3,0) -- (0,0);
\draw[ style=thick, color=blue ] (0,0) -- ({sqrt(4.5)},{sqrt(4.5)});
\fill[mypoints] (0,-2) circle (\ptsize);
\fill[mypoints] (-2,0) circle (\ptsize);
\fill[mypoints] ({sqrt(2)},{sqrt(2)}) circle (\ptsize);
\end{tikzpicture}
\hspace{0.2pc}
\begin{tikzpicture}[baseline=(current bounding box.center),scale=0.47]  
\tikzset{mypoints/.style={fill=white,draw=black,thick}}
\def\ptsize{2.5pt}
\clip (0,0) circle (3);
\draw[ ] (0,0) circle (2);
\draw[red, ultra thick] (0,0)+(180:2) arc (180:405:2); 
\draw[ style=thick, color=blue ] (0,-3) -- (0,3);
\fill[mypoints] (0,2) circle (\ptsize);
\fill[mypoints] (0,-2) circle (\ptsize);
\fill[mypoints] (-2,0) circle (\ptsize);
\fill[mypoints] ({sqrt(2)},{sqrt(2)}) circle (\ptsize);
\end{tikzpicture}
\hspace{0.2pc}
\begin{tikzpicture}[baseline=(current bounding box.center),scale=0.47]  
\tikzset{mypoints/.style={fill=white,draw=black,thick}}
\def\ptsize{2.5pt}
\clip (0,0) circle (3);
\draw[ ] (0,0) circle (2);
\draw[red, ultra thick] (0,0)+(90:2) arc (90:135:2); 
\draw[red, ultra thick] (0,0)+(270:2) arc (270:315:2);
\draw[ style=thick, color=blue ] (0,-3) -- (0,3);
\draw[ style=thick, color=blue ] ({-sqrt(4.5)},{sqrt(4.5)}) -- ({sqrt(4.5)},{-sqrt(4.5)});
\fill[mypoints] (0,-2) circle (\ptsize);
\fill[mypoints] (0,2) circle (\ptsize);
\fill[mypoints] ({-sqrt(2)},{sqrt(2)}) circle (\ptsize);
\fill[mypoints] ({sqrt(2)},-{sqrt(2)}) circle (\ptsize);
\end{tikzpicture}
\caption{{\small BNS invariants $\Sigma^1(G)$ (in red) and tropical curves 
$-\Trop(\VV^1(G))$ (in blue) for the one-relator groups $G$ from Examples \ref{ex:q-torus}, 
\ref{ex:deleq}, \ref{ex:brown}, and \ref{ex:dunfield}}} 
\label{fig:bns} 
\end{figure}
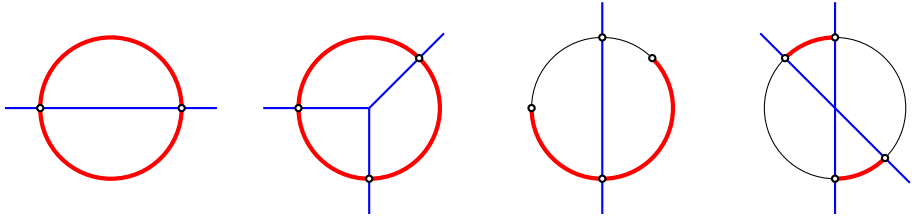

\begin{example}
\label{ex:q-torus}
Let 
$G=\langle x_1,x_2\mid x_1x_2^2x_1^{-1}x_2^{-2}\rangle$. Then  
$\Delta_G= t_2+1$ and 
$\Sigma^1(G)=S^1 \setminus \{(\pm 1,0)\}$.
\end{example}
  
\begin{example}
\label{ex:deleq}
Let $G=\langle x_1,x_2 \mid x_2^2 (x_1 x_2^{-1})^2 x_1^{-2}\rangle$. 
Then $\Delta_G=t_1+t_2+1$ and 
$\Sigma^1(G)=S^1 \setminus \{(-1,0), (0,-1), (\sfrac{1}{\sqrt{2}},\sfrac{1}{\sqrt{2}})\}$.  
\end{example}

In both of the above examples, 
$\Sigma^1(G)$ coincides with $-S(\Trop (\VV^1(G)))^{\compl}$, 
but is properly contained in 
$S(\tau^{\R}_1(\VV^1(G)))^{\compl}=S^1$.

\begin{example}
\label{ex:brown}
Let $G=\langle x_1,x_2 \mid x_1^{-1} x_2^{-1} x_1 x_2^2 x_1^{-1} x_2^{-1}
x_1^2 x_2^{-1} x_1^{-1} x_2 x_1^{-1} x_2 x_1 x_2^{-1}\rangle$.  
As an application of his algorithm,
Brown showed that $\Sigma^1(G)$ consists of two open arcs on the unit circle, 
joining the points $(-1,0)$  to $(0,-1)$ and $ (0,-1)$ to 
$(\sfrac{1}{\sqrt{2}},\sfrac{1}{\sqrt{2}})$, 
respectively. On the other hand, $\Delta_G=t_1-1$ and so 
$S(\Trop (\VV^1(G)))=S(\tau_1^{\R}(\VV^1(G)))$ 
consists of the whole circle with the points $(0,\pm 1)$ removed. Thus, 
 inclusion \eqref{eq:sigma1 bound} is strict in this case.
\end{example}

\section{Compact 3-manifolds}
\label{sect:3mfd}

In this section, $M$ will denote a compact, connected $3$-manifold 
with $b_1(M)>0$.  We say that a nontrivial cohomology class
$\phi\in H^1(M;\Z) = \Hom(\pi_1(M),\Z)$ is a fibered class if there exists
a fibration $p\colon M\to S^1$ such that the induced map
$p_*\colon \pi_1(M)\to \pi_1(S^1)=\Z$ coincides with $\phi$. 

The \emph{Thurston norm}\/ $\| \phi \|_T$ of a class 
$\phi\in H^1(M;\Z)$ is defined as the infimum of $-\chi(S_0)$, 
where $S$ runs though all the properly embedded, oriented 
surfaces in $M$  dual to $\phi$, and $S_0$ denotes 
the result of discarding all components of $S$ which 
are disks or spheres. In \cite{Th}, Thurston  proved that 
$\|-\|_T$ defines a seminorm on $H^1(M;\Z)$, which can 
be extended to a continuous seminorm on $H^1(M;\R)$.

The unit norm ball, $B_T=\{\phi\in H^1(M;\R) \mid \|\phi \|_T\leq 1\}$,
is a rational polyhedron with finitely many sides and which is symmetric 
in the origin. Moreover, there are facets of $B_T$, called the {\em fibered faces}\/ 
(coming in antipodal pairs), so that a class $\phi \in H^1(M;\Z)$ fibers if and 
only if it lies in the cone over the interior of a fibered face.
In \cite[Thm.~E]{BNS}, Bieri, Neumann, and Strebel showed that 
the BNS invariant of $G=\pi_1(M)$ is the projection onto $S(G)$ of 
the open fibered faces of the Thurston norm ball $B_T$; in particular, 
$\Sigma^1(G)=-\Sigma^1(G)$. 

Under some mild assumptions, McMullen \cite{McM} established an 
inequality between the Alexander and Thurston norms of a $3$-manifold 
$M$. When combined with the above theorem from \cite{BNS}, this 
inequality leads to an upper bound for the BNS invariant 
of $G$ in terms of the unit ball in the Alexander norm, see \cite{Dun}. 
We show next how to recover this upper bound from our tropical bound; 
for the sake of brevity, we restrict to the case $b_1(M)\ge 2$.  For 
more on the characteristic varieties of $3$-manifolds, as well as 
their (exponential) tangent cones, we refer to \cite{Su-tc3d}.

\begin{proposition}
\label{prop:sigma-3m}
Let $M$ be a compact, connected, orientable, $3$-manifold 
with empty or toroidal boundary. Set $G=\pi_1(M)$ and 
assume $b_1(M)\ge 2$. Then
\begin{enumerate}
\item \label{th1}
$\Trop\big(\VV^1(G)\cap \T_G^0\big)$ 
coincides with the positive-codimension skeleton of $\mathcal{F}(B_A)$, 
the face fan of the unit ball in the Alexander norm.
\item  \label{th2}
$\Sigma^1(G) $ is contained in the union of the open cones on 
the facets of $B_A$.
\end{enumerate}
\end{proposition}

\begin{proof}
Work of Milnor \cite{Mi62} and Turaev \cite{Tu75, Tu} shows that 
the Alexander polynomial of a $3$-manifold is symmetric (up to units).  
Moreover, as shown by McMullen \cite{McM} 
(see also \cite{EN,Tu}), we have that $I^p_H\cdot ( \Delta_{G} ) = E_1(A_G)$, 
where $H=G_{\ab}/\Tors$, and $p=1$ if $\partial M=\myempty$ 
and $p=2$, otherwise. Thus, claim \eqref{th1} follows from 
Proposition \ref{prop:trop-newt}, while claim  \eqref{th2} follows 
from Theorem \ref{thm:sigma1-G}.
\end{proof}

\begin{remark}
\label{rem:mcm}
Here is a sketch of the original proof of claim \eqref{th2}. 
As shown by McMullen in \cite{McM}, $\|\phi\|_A\le \|\phi\|_T$ 
for all $\phi\in H^1(M,\R)$, with equality if $\phi\in H^1(M,\Z)$ 
is a fibered class;  thus $B_T\subseteq B_A$, and 
each fibered face of $B_T$ is included in a face of $B_A$. 
Since $\Sigma^1(G)$ 
is the projection onto $S(G)$ of the fibered faces of $B_T$,  
the claim follows.
\end{remark}

\begin{remark}
\label{rem:poin3d}
Let $G$ be a $3$-dimensional Poincar\'{e} duality group such 
that $G$ admits a finite $K(G,1)$ and $\Z{G}$ embeds in a 
division algebra.  Under these hypotheses, 
Kielak shows in \cite[Thm.~5.32]{Ki19} 
that $\Sigma^1(G)$ is the projection onto $S(G)$ of the  
open cones on the marked faces of a certain polytope
in $H^1(G,\R)$. 
It would be interesting to know whether an analogue of 
Proposition \ref{prop:sigma-3m} still holds, 
with $B_T$ replaced by Kielak's marked polytope.
\end{remark} 

The next example shows that the tropical bound from Corollary \ref{cor:sigma1-X} 
(or Proposition \ref{prop:sigma-3m}) gives more information than the exponential 
tangent cone bound from Theorem \ref{thm:bns tau}.

\begin{example}
\label{ex:3mfd}
Let $M$ be a closed, orientable $3$-manifold with $H_1(M,\Z)=\Z^2$ 
and $\Delta_M=(t_1+t_2)(t_1t_2+1)-4t_1t_2$ (such a manifold exists 
by  \cite[VII.5.3]{Tu}).  It is readily seen that $\tau_{1}(\VV^1_1(M))=\{0\}$, 
yet $\Trop(V(\Delta_M))$ consists of the two diagonals in the plane.  Thus, 
$\Sigma^1(M)$ is contained in the complement in $S^1$ of the four points 
$(\pm \sfrac{1}{\sqrt{2}},\pm \sfrac{1}{\sqrt{2}})$ and 
$(\pm \sfrac{1}{\sqrt{2}},\mp \sfrac{1}{\sqrt{2}})$.
\end{example}

We conclude this section with a couple of examples from link theory.

\begin{example}
\label{ex:Hopf link}
Let $L$ be the $n$-component Hopf link, consisting of $n\ge 2$ fibers 
of the Hopf fibration $S^3\to S^2$, and let $G$ be the link group. 
Then $\Delta_G=(t_1\cdots t_n - 1)^{n-2}$, 
and a quick computation shows that $S(\Trop(\VV^1(G)))=S(\tau^\R_1(\VV^1(G)))$ 
is the great sphere cut out by the hyperplane $\sum x_i=0$, while 
$\Sigma^1(G)=S^{n-1}\setminus S^{n-2}$ is the complement of 
this great sphere. 
\end{example}

As shown by Dunfield \cite{Dun}, though, the fibered faces of $B_T$ may 
be strictly included in the corresponding faces of $B_A$.

\begin{example}
\label{ex:dunfield} 
Let $L$ be the $2$-component studied in \cite[\S 6]{Dun}. 
The link group $G$ is a $2$-generator, $1$-relator group with 
$\Delta_G=(t_1-1)(t_1t_2-1)$.  Thus, $S(\Trop(\VV^1(G)))=S(\tau^\R_1(\VV^1(G)))$ 
consists of the four points $(\pm 1,0)$ and 
$(\pm \sfrac{1}{\sqrt{2}},\mp \sfrac{1}{\sqrt{2}})$. On the other hand, 
$\Sigma^1(G)$ consists of two open arcs joining those two pairs of points 
(see the last picture in Figure~\ref{fig:bns}).   
\end{example} 

\section{Orbifold Riemann surfaces}
\label{sect:orbifolds} 

Let $\Sigma_g$ be a Riemann surface of genus $g\ge 1$ and let 
$k\ge 0$ be an integer. If $k>0$, fix points $q_1,\dots ,q_k$ in $\Sigma_g$, 
assign to these points an integer weight vector 
$\mathbf{m}=(m_1,\dots, m_k)$ with $m_i\ge 2$, and set $\abs{\mathbf{m}}=k$.  
The orbifold Euler characteristic of the surface with marked points is defined as 
$\chi^{\orb}(\Sigma_g, \mathbf{m})=2-2g-\sum_{i=1}^{k} (1-1/m_i)$, while 
the orbifold fundamental group $\Gamma=\pi_1^{\orb}(\Sigma_g, \mathbf{m})$ 
has presentation 
\begin{equation}
\label{eq:orbipi1}
\Gamma=
\left\langle 
\begin{array}{c}
x_1,\dots, x_g, y_1,\dots , y_g, \\[2pt]
z_1, \dots ,z_k 
\end{array}\left|
\begin{array}{c}
[x_1,y_1]\cdots [x_g,y_g] z_1\cdots z_k=1, \\[2pt]
z_1^{m_1}=\cdots =z_{k}^{m_k}=1
\end{array}
\right.
\right\rangle. 
\end{equation}

There is an obvious epimorphism 
$\Gamma\surj \pi_1(\Sigma_g)$, 
obtained by sending each $z_i$ to $1$.  Upon abelianizing, 
we obtain an isomorphism 
$\Gamma_{\ab} \cong \pi_1(\Sigma_g)_{\ab} \oplus \Theta$, 
where the torsion subgroup 
\begin{equation}
\label{eq:tors-g}
\Theta\coloneqq\Tors(\Gamma_{\ab})=
\Z_{m_1}\oplus \cdots \oplus \Z_{m_k}/(1,\dots ,1)\, , 
\end{equation}
has order 
$\abs{\Theta}=\theta(\mathbf{m})$, where 
\begin{equation}
\label{eq:theta}
\theta(\mathbf{m})\coloneqq m_1\cdots m_k/\lcm(m_1,\dots ,m_k)\, .
\end{equation}
Let us identify the character group $\T_\Gamma=\Hom(\Gamma, \C^{\times})$ 
with $\T_\Gamma^{0} \times \T_{\Theta}$, where 
$\T_\Gamma^{0}\cong (\C^{\times})^{2g}$ and $\T_{\Theta} \cong \Theta$.
As noted in \cite{ACM} (see also \cite{Su-imrn}), we then have that 
\begin{equation}
\label{eq:v1piorb}
\VV^1(\Gamma)=\begin{cases}
\T_\Gamma & \text{if $g>1$}, \\[2pt]
\big( \T_\Gamma \setminus 
\T_\Gamma^{0}\big) \cup  \{1\} 
& \text{if $g=1$ and $\theta(\mathbf{m})> 1$}, \\[2pt]
\{1\}  
& \text{otherwise}\, .
\end{cases}
\end{equation}

Let us identify $\Hom(\Gamma,\R)$ with $\R^{2g}$.  In the first case, 
$\Trop(\VV^1(\Gamma))=\tau_1(\VV^1(\Gamma))=\R^{2g}$, while in the 
third case $\Trop(\VV^1(\Gamma))=\tau_1(\VV^1(\Gamma))=\{0\}$.  On the 
other hand, in the second case $\T_\Gamma \setminus \T_\Gamma^{0}\cong 
\T_\Gamma^{0} \times (\Theta\setminus \{1\})$, where the second factor is 
nonempty; thus, $\Trop(\VV^1(\Gamma))=\R^{2g}$ 
yet $\tau_1(\VV^1(\Gamma))=\{0\}$. Using Corollary \ref{cor:sigma-empty-2}, 
we obtain the following result.

\begin{proposition}
\label{prop:bns-gamma}
Let $(\Sigma_{g},\mathbf{m})$ be a compact orbifold surface such that 
either $g\ge 2$ or $g=1$ and $\theta(\mathbf{m})>1$. Then
$\Sigma^1(\pi_1^{\orb}(\Sigma_g, \mathbf{m}))=\myempty$.
\end{proposition}

Now let $\Sigma_{g,r}=\Sigma_g \setminus \{p_1,\dots , p_r\}$ be 
a Riemann surface of genus $g\ge 0$ with $r$ points removed ($r\ge 1$).   
Then $\pi_1(\Sigma_{g,r})=F_n$, where $n=b_1(\Sigma_{g,r})=2g+r-1$.  
To avoid trivialities, we will assume $n>0$; note that $n=1$ if and only if 
$g=0$ and $r=2$, i.e., $\Sigma_{g,r}=\C^{\times}$.

Next, consider a non-compact $2$-orbifold, i.e., 
a surface $(\Sigma_{g,r},\mathbf{m})$ with marked points $q_1,\dots, q_k$, 
weight vector $\mathbf{m}=(m_1,\dots, m_k)$, and $r$ points removed. 
The orbifold Euler characteristic of the surface is defined as 
$\chi^{\orb}(\Sigma_{g,r}, \mathbf{m})=1-n-\sum_{i=1}^{k} (1-1/m_i)$, 
while the orbifold fundamental group 
$\Gamma=\pi_1^{\orb}(\Sigma_{g,r}, \mathbf{m})$ 
associated to these data is the free product 
\begin{equation}
\label{eq:orbipi1 noncompact}
\Gamma=
F_{n} * \Z_{m_1} * \cdots * \Z_{m_k}\, .
\end{equation}
(The case $n=k=1$ and $m_1=2$ is the one analyzed in Example \ref{ex:z*z2}.) 
Thus, $\Gamma_{\ab}=\Z^{n}\oplus \Theta$, 
where now $\Theta=\Z_{m_1}\oplus \cdots \oplus \Z_{m_k}$.  
A similar computation as above shows that
\begin{equation}
\label{eq:v1pi}
\VV^1(\Gamma)=\begin{cases}
\T_\Gamma & \text{if $n>1$}, \\
\big( \T_\Gamma\setminus \T_\Gamma^{0}\big) \cup \{1\}
& \text{if $n=1$ and $\abs{\mathbf{m}}>0$}, \\
\{1\} & \text{if $n=1$ and $\abs{\mathbf{m}}=0$}. 
\end{cases}
\end{equation}

Identifying $\Hom(\Gamma,\R)=\R^{n}$, 
we find that $\Trop(\VV^1(\Gamma))=\tau_1(\VV^1(\Gamma))=\R^{n}$ 
in the first case and $\Trop(\VV^1(\Gamma))=\tau_1(\VV^1(\Gamma))=\{0\}$ 
in the last one.  On the other hand, in the second case 
$\Trop(\VV^1(\Gamma))=\R^{n}$ yet $\tau_1(\VV^1(\Gamma))=\{0\}$. 
Applying Corollary \ref{cor:sigma-empty-2} once again, we obtain the following 
result.

\begin{proposition}
\label{prop:bns-gamma-punc}
Let $(\Sigma_{g,r},\mathbf{m})$ be a punctured orbifold surface with 
negative orbifold Euler characteristic. Then 
$\Sigma^1(\pi_1^{\orb}(\Sigma_{g,r}, \mathbf{m}))=\myempty$.
\end{proposition}

\section{Seifert manifolds}
\label{sect:seifert}

A compact $3$-manifold is a Seifert fibered space if it is 
foliated by circles. One can think of such a manifold $M$ as a  
bundle in the category of orbifolds, in which the circles of the 
foliation are the fibers, and the base space of the orbifold bundle is the 
quotient space of $M$ obtained by identifying each circle to a point. 
We point to \cite{JW,NeumannRaymond, Scott} as general references 
for the subject. 

For our purposes here, we will only consider compact, connected, 
orientable Seifert manifolds with orientable base.  Every such 
manifold $M$ admits an effective circle action, with orbit space 
a Riemann surface $\Sigma_g$, and finitely many exceptional 
orbits, encoded by pairs of coprime integers $(\alpha_1,\beta_1), 
\dots,  (\alpha_r,\beta_r)$ with $\alpha_i>0$. The fundamental group  
$G=\pi_1(M)$ admits a presentation of the form 
\begin{equation}
\label{eq:seifert-pi1}
G=
\left\langle 
\begin{array}{c}
x_1,\dots, x_g, y_1,\dots , y_g, \\[2pt]
z_1,\dots,z_r, h
\end{array}\left|
\begin{array}{l}
 [x_1,y_1]\cdots[x_g,y_g]z_1\cdots z_r=1, \\[2pt]
 z_1^{\alpha_1}h^{\beta_1}=\cdots =  z_r^{\alpha_r}h^{\beta_r} =1, 
\ \text{$h$ central}
\end{array}
\right.
\right\rangle, 
\end{equation}
where $h$ is the homotopy class of a regular orbit. 
We will denote by $e=-\sum_{i=1}^{r}\beta_i/\alpha_i$ 
the Euler number of the orbifold bundle. 
If $e=0$, then $G_{\ab}=\Z^{2g+1}$, and  if $e\ne 0$, then  
$G_{\ab}=\Z^{2g} \oplus \Tors(G_{\ab})$, where 
$\abs{ \Tors(G_{\ab})}=\alpha_1\cdots \alpha_r \abs{e}$. 
Write $\boldsymbol{\alpha}=(\alpha_1,\dots , \alpha_r)$. 

\begin{proposition}
\label{prop:cv1-seifert}
Let $M$ be Seifert manifold with genus $g$ base surface, 
$r$ exceptional fibers, and orbifold Euler number $e$. 
Suppose $e\ne 0$ and either $g>1$, or $g=1$ and 
$\theta(\boldsymbol{\alpha})>1$. Then $\Sigma^1(M)=\myempty$.
\end{proposition}

\begin{proof}
The orbit map $p\colon M\to \Sigma_g$ induces an epimorphism 
$\pi_1(M)\surj \pi_1(\Sigma_g)$ sending $h\mapsto 1$. This 
map factors through an epimorphism $G\surj \Gamma$, where 
$\Gamma=\pi_1^{\orb}(\Sigma_g, \boldsymbol{\alpha})$. 
Passing to abelianizations, we obtain an epimorphism 
$p_{*}\colon G_{\ab}\surj \Gamma_{\ab}$, which in turn 
induces a monomorphism $p^*\colon \T_\Gamma\inj \T_G$. 

Since $e\ne 0$, the groups $G_{\ab}$ and $\Gamma_{\ab}$ 
have the same rank (equal to $2g$); thus, the map $p^*$ 
restricts to an isomorphism $p^* \colon \T_\Gamma^0\isom \T_G^0$.
On the other hand, we also know from Proposition \ref{prop:v1-nat} 
that the map  $p^*\colon \T_\Gamma\inj \T_G$ sends 
$\VV^1(\Gamma)$ to $\VV^1(G)$.  There are two 
cases to consider.

\begin{itemize}
\item
If $g>1$, then by \eqref{eq:v1piorb}, $\VV^1(\Gamma)=\T_{\Gamma}$. 
Therefore $\VV^1(G)$ contains $p^*(\T_{\Gamma})=p^*(\VV^1(\Gamma))$,
and thus it must also contain $p^*(\T_{\Gamma}^0)=\T_G^0$. 
\item
If $g=1$ and $\theta(\boldsymbol{\alpha})>1$, then again by \eqref{eq:v1piorb}, 
$\VV^1(\Gamma)=\T_\Gamma \setminus \T_\Gamma^{0}$, which 
is a non-empty union of torsion-translated copies of $\T_\Gamma^{0}$, 
since $\Tors(\Gamma_{\ab})\ne 0$. Arguing as above, we  infer that $\VV^1(G)$ 
must contain a translated copy of $\T_G^0$. 
\end{itemize}
In either case, the variety $\VV^1(G)$ contains a connected component 
of $\T_G$. Hence, by Corollary \ref{cor:sigma-empty-2}, $\Sigma^1(M)=\myempty$.
\end{proof}

A well-known class of Seifert manifolds arises in singularity theory. 
Let $(a_1, \dots, a_n)$ be an $n$-tuple of integers,
with $a_j \geq 2$. Consider the variety $X$ in $\C^n$ defined by the
equations $c_{j 1}x^{a_1}_1 + \cdots + c_{j n}x^{a_n}_n=0$,
for $1 \leq j \leq n-2$. Assuming all maximal minors of the
matrix $\big(c_{j k}\big)$ are non-zero, $X$ is a
quasi-homogeneous surface, with an isolated singularity at $0$.
The space $X$ admits a good $\C^{\times}$-action.
Set $X^*=X\setminus \{0\}$, and let $p\colon X^* \to\Sigma_g$ be the
corresponding projection onto a smooth projective curve.

By definition, the Brieskorn manifold $M= \Sigma (a_1, \dots, a_n)$
is the link of the quasi-ho\-mogenous singularity $(X, 0)$.  
As such, $M$ is a closed, smooth, oriented $3$-manifold which is 
homotopy equivalent to $X^*$. In fact, $M$ is a Seifert 
manifold, with $S^1$-action obtained by restricting the 
$\C^{\times}$-action on $X$. Put
$\ell=\lcm(a_1, \dots, a_n)$, 
$\ell_i=\lcm (a_1, \dots,\widehat{a_i} ,\dots,a_n)$, 
and $a=a_1 \cdots a_n$, and note that $s_i\coloneqq a / (a_i\ell_i)\in \Z$. 
Set $\alpha_i=\ell/\ell_i$ and define integers $0<\beta_i<\alpha_i$ by 
$\beta_i (\ell/a_i)\equiv 1\pmod{\alpha_i}$. The $S^1$-equivalent
homeomorphism type of $M$ is then determined by
the following Seifert invariants associated to the
projection $\left. p\right|_M\colon M \to \Sigma_g$:
\begin{itemize}
\item The exceptional orbit data,
$(s_1(\alpha_1, \beta_1), \ldots, s_n(\alpha_n, \beta_n))$, 
where $s_i(\alpha_i ,\beta_i)$
signifies $(\alpha_i ,\beta_i)$ repeated $s_i$ times,
unless $\alpha_i=1$, in which case $s_i(\alpha_i ,\beta_i)$
is to be removed from the list.
\item The genus of the base curve, 
$g=\frac{1}{2}\left( 2+\frac{(n-2)a}{\ell}-\sum^n _{i=1}s_i \right)$.
\item The orbifold Euler number,  $e=-a/\ell^2<0$.
\end{itemize}

The group $H_1(M, \Z)$ has rank $2g$ and torsion subgroup 
of order $\alpha_1^{s_1} \cdots \alpha_n^{s_n} \cdot \abs{e}$.
Set $\alpha=\alpha_1^{s_1} \cdots \alpha_n^{s_n} / 
\lcm(\alpha_1,\dots ,\alpha_n)$.  Applying 
Proposition \ref{prop:cv1-seifert}, we obtain the following corollary.

\begin{corollary}
\label{cor:sigma-brieskorn}
Let $M$ be a Brieskorn manifold as above. If either $g>1$, 
or $g=1$ and $\alpha>1$, then $\Sigma^1(M)=\myempty$.
\end{corollary}

Here is a concrete example, which builds on computations 
from  \cite{DPS-imrn, SYZ-pisa}.

\begin{example}
\label{ex:sigma248}
The Brieskorn manifold $M=\Sigma(2,4,8)$ fibers over the torus 
with two exceptional fibers, each of multiplicity $2$, and with orbifold 
Euler number $e=-1$; thus, $g=1$ and $\alpha=2$.  
By Corollary \ref{cor:sigma-brieskorn}, 
we have that $\Sigma^1(M)=\myempty$.
Alternatively, it is readily seen that $H_1(M,\Z)=\Z^2 \oplus \Z_4$,  
and hence $\T_M= (\C^{\times})^2 \times \{\pm 1, \pm i\}$. 
A short computation shows that $\VV^1(M) = \{1\} \cup  (\C^{\times})^2 \times \{-1\}$. 
Thus, $\tau_1(\VV^1(M))=\{0\}$, and so the 
bound from Theorem \ref{thm:bns tau} does not say anything in this case. 
On the other hand, $\Trop_1(\VV^1(M))=H^1(M,\R)$, and so the 
bound from Corollary \ref{cor:sigma1-X}  gives the precise 
answer for $\Sigma^1(M)$.
\end{example}

\section{K\"{a}hler manifolds}
\label{sect:kahler} 
Let $M$ be a compact, connected, K\"{a}hler manifold.    
The structure of the characteristic varieties of such manifolds 
was determined by Green and Lazarsfeld in \cite{GL87, GL91}, 
building on work of Castelnuovo and de Franchis, 
Beauville \cite{Be92}, and Catanese \cite{Cat91}.   
For this reason, the varieties $\VV^i(M)$ are also known 
in this context as the {\em Green--Lazarsfeld sets}\/ of $M$. 
The theory was further amplified by Simpson \cite{Si93}, 
Ein and Lazarsfeld \cite{EL97},  and Arapura \cite{Ar97}, 
developed in  \cite{Di07, De08, DPS-duke, 
ACM, BW12}, and completed by Wang in \cite{Wa16}.  
The cornerstone of this theory is that all the Green--Lazarsfeld 
sets of $M$ are finite unions of torsion translates of algebraic 
subtori of $H^1(M,\C^{\times})$. 

In degree $i=1$, the structure of the Green--Lazarsfeld 
set $\VV^1(M)$ can be made more precise. 
As before, let $\Sigma_g$ be a Riemann surface of genus 
$g\ge 1$, with marked points $q_1,\dots, q_k$, and weight 
vector $\mathbf{m}=(m_1,\dots, m_k)$ with $m_i\ge 2$, 
and set $\abs{\mathbf{m}}\coloneqq k$.   
A surjective map $f\colon M\to (\Sigma_g,\mathbf{m})$ 
is called an {\em orbifold fibration}\/ if $f$ is holomorphic, 
the fiber over any non-marked point is connected, 
and, for every marked point $q_i$ the multiplicity of the fiber 
$f^{-1}(q_i)$ equals $m_i$.  Such a map induces 
an epimorphism $f_{\sharp} \colon \pi_1(M) \surj \Gamma$, 
where $\Gamma=\pi_1^{\orb}(\Sigma_g, \mathbf{m})$. 
The induced morphism of character groups, 
$f_{\sharp}^*\colon \T_\Gamma\inj \T_{\pi_1(M)}$, 
sends $\VV^1(\Gamma)$ to a union of (possibly torsion-translated) 
subtori inside $\VV^1(M)$.

Two orbifold fibrations, $f\colon M\to (\Sigma_g,\mathbf{m})$ and 
$f'\colon M\to (\Sigma_{g'},\mathbf{m}')$, are equivalent 
if there is a biholomorphic map $h\colon \Sigma_g\to \Sigma_{g'}$ 
which sends marked points to marked points, while preserving multiplicities. 
As shown by Delzant \cite[Thm.~2]{De08}, a K\"ahler manifold $M$ admits 
only finitely many equivalence classes of orbifold fibrations  
for which the orbifold Euler characteristic of the base is negative.
The next theorem, which summarizes several results 
from \cite{Ar97, Di07, De08, ACM}, shows that 
{\em all}\/ positive-dimensional components in the first 
characteristic variety of $M$ arise by pullback 
along this finite set of  orbifold fibrations.  

\begin{theorem}
\label{thm:cv1 kahler}
Let $M$ be a compact K\"{a}hler manifold.  Then 
\begin{equation}
\label{eq:vvk}
\VV^1(M)=\bigcup_{\alpha \in \operatorname{Fib}(M)} (f_{\alpha})_{\sharp}^*
\big(\VV^1(\pi_1^{\orb}(\Sigma_{g_\alpha},\mathbf{m}_{\alpha}) \big) 
\cup Z\, , 
\end{equation}
where $Z$ is a finite set of torsion characters, and the union runs 
over a (finite) set of equivalence classes of orbifold fibrations 
$f_{\alpha}\colon M\to (\Sigma_{g_\alpha},\mathbf{m}_{\alpha})$ 
with either $g_{\alpha}\ge 2$, or $g_{\alpha}=1$ and 
$\abs{\mathbf{m}_{\alpha}}\ge 2$. 
\end{theorem}

It now follows from Theorem \ref{thm:sigma-pencils} and 
formula \eqref{eq:v1piorb} that 
\begin{equation}
\label{eq:sigma-k-inc}
\Sigma^1(M)\subseteq \bigg(
 \,\medmath{\bigcup_{\alpha \in \operatorname{Fib}(M)}}\, S( f_{\alpha}^* 
( H^1(C_{\alpha}, \R)) ) \bigg)^{\compl}.
\end{equation}
Remarkably, a result of Delzant (\cite[Thm.~1.1]{De10}) shows 
that the above inclusion holds as an equality.  In view of 
Corollary \ref{cor:trop trans}, we may recast this result in 
the tropical setting, as follows.

\begin{theorem}[\cite{De10}]
\label{thm:bns kahler}
Let  $M$ be a compact K\"{a}hler manifold. 
Then 
\begin{equation}
\label{eq:delz}
\Sigma^1(M)=S\big(\!\Trop(\VV^1(M))\big)^{\compl}.
\end{equation}
\end{theorem}

\begin{remark}
\label{rem:fv}
As shown by Friedl and Vidussi in \cite[Lemma 2.3]{FV19}, the set $\Sigma^1(M)$ 
is non-empty if and only if $\pi_1(M)$ virtually algebraically fibers 
(i.e., it contains a finite-index subgroup which maps onto $\Z$ 
with finitely generated kernel). By Theorem \ref{thm:bns kahler}, 
this condition is equivalent to $\Trop(\VV^1(M))\subsetneqq H^1(M,\R)$.
\end{remark}

\begin{example}
\label{ex:omega kahler}
Let $C_1$ be a smooth, complex curve of genus $2$ 
with an elliptic involution $\sigma_1$, and let $C_2$ be 
a curve of genus $3$ with a free involution $\sigma_2$. 
Then $\Sigma_1=C_1/\sigma_1$ is a curve of genus one, 
$\Sigma_2=C_2/\sigma_2$ is a curve of genus two, 
and $M=(C_1 \times C_2)/\sigma_1 \times \sigma_2$   
is the smooth, complex projective surface constructed in \cite{CCM}.
Projection onto the first coordinate yields an orbifold fibration  
$f_1 \colon M\to \Sigma_1$ with two multiple fibers, 
each of multiplicity $2$, while 
projection onto the second coordinate 
defines a smooth fibration $f_2\colon M\to \Sigma_2$. 
It is readily seen that $H_1(M,\Z)=\Z^6$, and so  
$H^1(M,\C^{\times}) = (\C^{\times})^6$.   In  \cite{Su-imrn}, 
we showed that   $\VV^1(M)=
\{t_4=t_5=t_6=1, \, t_3=-1\}\cup \{ t_1=t_2=1\}$, 
with the two components obtained by pullback along 
the maps $f_1$ and $f_2$.   Thus, $\Sigma^1(M)$ 
is the complement in $S^5$ of the two subspheres 
cut out by the subspaces $\{x_3=\cdots = x_6=0\}$ and 
$\{x_1= x_2=0\}$, respectively.
\end{example}

\section{Hyperplane arrangements}
\label{sect:arrs}

\subsection{Arrangement complements}
\label{subsec:hyparr}
Let $\A$ be an arrangement of hyperplanes in the complex affine space  
$\C^{\ell}$. We denote by $L(\A)$ the intersection poset of $\A$, and by 
$L_k(\A)$ those subspaces  in $L(\A)$ of codimension $k$. 
We also let $M(\A)=\C^{\ell}\setminus \bigcup_{H\in \A} H$ be the complement 
of the arrangement. Without much loss of generality, we may assume $\A$ is 
central and essential, that is, $\bigcap_{H\in \A} H=\{0\}$.    

The complement $M=M(\A)$ has the homotopy type of a finite, 
connected CW-complex of dimension $\ell$. The homology 
groups $H_i(M,\Z)$ are all torsion-free, and $b_1(M)$ is 
equal to $\abs{\A}$, the number of hyperplanes in $\A$.  
The cohomology ring of $M$ was computed by Brieskorn in the 
early 1970s, building on work of Arnol'd; a description of this 
ring in terms of $L(\A)$ was given by Orlik and Solomon in 1980. 
Brieskorn's work  implies that the space 
$M$ is formal; thus, by \eqref{eq:tcone}, we have that 
$\tau_1(\VV^q(M))=\RR^q(M)$, for all $q$. 

\subsection{BNSR invariants of arrangements}
\label{subsec:bnsr-arr}
We now turn to the $\Sigma$-invariants of arrangement complements.  
Although we do not know whether these sets are symmetric in 
general, the next result (which generalizes an observation from 
\cite{PS-plms}) singles out a situation where they are.

\begin{proposition}
\label{props:arr-symm}
Let $\A$ be the complexification of an arrangement $\A_{\R}$ of real 
hyperplanes in $\R^{\ell}$, and let $M$ be the complement 
of $\A$ in $\C^{\ell}$.  Then $\Sigma^q(M,\Z)=-\Sigma^q(M,\Z)$, for all $q$.
\end{proposition}

\begin{proof}
The group $H_1(M,\Z)$ is freely generated by the meridians 
of the hyperplanes in $\A$. If $\A$ is defined by real equations, 
complex conjugation in $\C^{\ell}$ restricts to a diffeormorphism $f\colon 
M\to M$ sending each meridian $S^1$ to itself by a map of degree $-1$. 
Thus, $f_*=-\id\colon H_1(M,\Z) \to H_1(M,\Z)$, and the claim follows 
from Proposition \ref{prop:symm}.
\end{proof}

At a 2007 Oberwolfach Mini-Workshop \cite{FSY}, the question 
was raised whether $\Sigma^1(M)$ equals $S(\RR^1(M,\R))^{\compl}$ 
for $M=M(\A)$. In an important particular case, the answer turned out to be yes. 
 
\begin{example}
\label{ex:braid}
 Let $\A$ be the braid arrangement in $\C^n$, consisting of all the 
hyperplanes $z_i-z_j=0$ with $1\le i<j\le n$.  The complement $M$ 
is the configuration space of $n$ ordered points in $\C$, and is 
thus a classifying space for the pure braid group $P_n$.  
As shown in \cite{KMM}, 
$\Sigma^1(M)$ is equal to $S(\RR^1(M,\R))^{\compl}$, which is the 
complement of a collection of $\binom{n+1}{4}$ great circles  
in $S^{\binom{n}{2}-1}$.
\end{example}

In general, though, the answer to the aforementioned question is no: as 
we showed in \cite{Su-pisa12}, there is an arrangement complement $M$ for 
which $\Sigma^1(M)\subsetneqq S(\RR^1(M,\R))^{\compl}$.  We will 
revisit this example in Proposition \ref{prop:del-sigma}, and explain it from 
a tropical point of view. 

\subsection{A tropical upper bound}
For a central, essential arrangement $\A$ in $\C^{\ell}$ with complement 
$M=M(\A)$, setting $n=\abs{\A}$ and ordering the hyperplanes as 
$H_1,\dots, H_n$ allows us to identify $H^1(M,\R) = \R^n$ and $S(M)=S^{n-1}$.
Since arrangement complements are formal, the ``resonance 
upper bound" from \eqref{eq:bns-res bound} holds. The next 
theorem provides a much improved  ``tropical  upper bound" 
for the BNSR invariants of arrangements.

\begin{theorem}
\label{thm:bnsr-trop-arr}
Let $M$ be the complement of an arrangement of $n$ hyperplanes in 
$\C^{\ell}$.  Then, for each $1\le q\le \ell-1$, the following hold.
\begin{enumerate}
\item \label{arr1}
$\Trop(\VV^q(M))$ is the union of a subspace arrangement in $\R^n$.
\item \label{arr2}
$\Sigma^q(M,\Z) \subseteq  S(\Trop(\VV^q(M)))^{\compl}$.
\end{enumerate}
Consequently, each BNSR invariant of $M$ is contained in the complement 
of an arrangement of great subspheres in $S^{n-1}$.
\end{theorem} 

\begin{proof}
A theorem of Arapura \cite{Ar97} implies that all the irreducible components of 
$\VV^q(M)$ are translated subtori of $(\C^{\times})^n$.  The first claim then  
follows from formulas \eqref{eq:tropt} and \eqref{eq:tropzt}.

As shown in \cite{DSY}, $M$ is an abelian duality 
space, and consequently, its characteristic varieties propagate, 
that is, $\VV^1(M)\subseteq\cdots\subseteq\VV^{\ell-1}(M)$.  
Therefore, $\VV^{\le q}(M)=\VV^q(M)$. The second claim 
now follows from Theorem \ref{thm:bns-trop}.  
\end{proof}

\subsection{Lower bounds for the \texorpdfstring{$\Sigma$}{Sigma}-invariants}
\label{subsec:lower}
Let $G=\pi_1(M)$.  In \cite{KP15}, Kohno and 
Pajitnov showed that the homology groups of $M$ with coefficients in the 
Novikov--Sikorav completion $\widehat{\Z{G}}_{\chi}$ vanish in 
degrees $i<\ell$, provided $\chi$ belongs to the positive orthant  
of $S^{n-1}$. In view of Theorem \ref{thm:bns novikov}, this implies that 
 the negative orthant, $S_{-}^{n-1}$, is contained in $\Sigma^q(M,\Z)$, 
for all $q<\ell$. In particular, if $\ell>1$, then $S_{-}^{n-1}\subseteq \Sigma^1(M)$.  
In fact, as shown in \cite{Su-pisa12}, 
\begin{equation}
\label{eq:sigma1 arr}
\Sigma^1(M)=S^{n-1}\setminus 
\big(S^{n-2}\setminus \Sigma^{1}(U)\big)\, , 
\end{equation}
where $S^{n-2}$ is the great sphere cut out by the hyperplane 
$\sum \chi_j =0$, and $U$ is 
the complement of the projectivized arrangement. In particular, 
$S^{n-1}\setminus S^{n-2}\subseteq \Sigma^1(M)$. In the case 
of an arrangement of $n$ lines through the origin of $\C^2$, the 
complement deform-retracts onto the complement of the $n$-component 
Hopf link from Example \ref{ex:Hopf link}. This lower bound coincides 
with the upper bound from Theorem \ref{thm:bnsr-trop-arr}, 
and so $\Sigma^1(M)=S^{n-1}\setminus S^{n-2}$.  In general, though, 
the inclusion is strict. 

\subsection{BNS invariants of arrangement groups} 
\label{subsec:bns-arr}
To see why this is the case, we need a more detailed description 
of the first characteristic variety of a hyperplane arrangement.  The 
notion of orbifold fibration   
is defined for arrangement complements (and, indeed, for all 
smooth, quasi-projective varieties) in exact analogy  
with K\"ahler manifolds (see \S\ref{sect:kahler}). The 
next theorem summarizes several results from 
\cite{Ar97, Di07, FY07, PeY, Yu, ACM} in this direction. 

\begin{theorem}
\label{thm:v1a}
The decomposition into irreducible components of the first characteristic 
variety of an arrangement complement $M$ is of the form 
\begin{equation}
\label{eq:vv-arr}
\VV^1(M)=\bigcup_{\alpha\in \operatorname{Fib}(M)} (f_{\alpha})_{\sharp}^*
\big(\VV^1(\pi_1^{\orb}(\Sigma_{0,s_\alpha},\mathbf{m}_{\alpha}) \big) 
\cup Z\, , 
\end{equation}
where $Z$ is a finite set of torsion characters, and the union runs 
over a finite set of equivalence classes of orbifold fibrations 
$f_{\alpha}\colon M\to (\Sigma_{0,s_\alpha},\mathbf{m}_{\alpha})$ 
for which either $s_{\alpha}=3$ or $4$, or 
$s_{\alpha}= 2$ and $\abs{\mathbf{m}_{\alpha}}>0$.
\end{theorem}

By formula \eqref{eq:v1pi}, each positive-dimensional component $T_{\alpha}$ 
in \eqref{eq:vv-arr} is a torus of dimension $s_{\alpha}-1$, translated by a non-trivial 
torsion character if $\abs{\mathbf{m}_{\alpha}}>0$. As shown by Falk and Yuzvinsky 
in \cite{FY07}, each component $T_{\alpha}$  passing through the origin 
arises from a {\em $k$-multinet}\/ on a sub-arrangement 
$\B\subseteq \A$ with $k=s_{\alpha}$. Such a multinet consists 
of a partition $\B=(\B_1,\dots, \B_k)$, together with 
multiplicities $m_H$ attached to each hyperplane $H$ and a subset 
$\XX\subset L_2(\A)$ such that several conditions are satisfied; most 
importantly, for each $X\in\XX$, the sum 
$n_X\coloneqq \sum_{H\in\B_i\colon H\supset X} m_H$ is independent of~$i$. 
For instance, each flat $X\in L_2(\A)$ of size $k\ge 3$ gives rise 
to a $k$-multinet on the corresponding ``local" sub-arrangement,
and thus, to a component $T_X\subset \VV^1(M)$ 
of dimension $k-1$. On the other hand, 
the non-local components of $\VV^1(M)$ passing through $1$ all 
have dimension $2$ or $3$, see \cite{PeY, Yu}. 

As a consequence,  we have the following description of $\Sigma^1(M)$ in 
terms of the orbifold fibrations supported by $M$, with negative orbifold Euler 
characteristic of the base.

\begin{theorem}
\label{thm:bns-arr}
Let $M=M(\A)$ be an arrangement complement. Then,
\begin{equation}
\label{eq:sigma-pen-arr}
\Sigma^1(M)\subseteq \Bigg( \bigcup_{\alpha \in \operatorname{Fib}(M)}\, 
S\Big( f_{\alpha}^* \big( H^1(\Sigma_{0,s_\alpha}, \R)\big) \Big) \Bigg)^{\compl}.
\end{equation}
\end{theorem}

\begin{proof}
By Theorem \ref{thm:v1a}, there is a finite set, $\operatorname{Fib}(M)$, 
indexing equivalence classes of orbifold fibrations,  
$f_{\alpha}\colon M\to (\Sigma_{0,s_\alpha},\mathbf{m}_{\alpha})$, 
for which the induced homomorphism from $G=\pi_1(M)$ to 
$G_{\alpha} =\pi_1^{\orb}(\Sigma_{0,s_\alpha},\mathbf{m}_{\alpha})$ 
is surjective. From the definition of this set, we also 
know that $\chi^{\orb}(\Sigma_{0,s_\alpha},\mathbf{m}_{\alpha})<0$, 
for each $\alpha$. 
Thus, by formula \eqref{eq:v1pi}, $\VV^1(T_{G_\alpha})$ is equal to either 
$\T_{G_{\alpha}}$ or $\T_{G_{\alpha}}\setminus \T^0_{G_{\alpha}}$.
In either case, $\VV^1(T_{G_\alpha})$  contains a connected component 
of $\T_{G_{\alpha}}$. Hence, Theorem \ref{thm:sigma-pencils} applies, 
and gives the desired conclusion.
\end{proof}

\subsection{Pointed multinets}
\label{subsec:pointed}
Following \cite{DeS14}, we describe a combinatorial 
construction which produces translated subtori 
in the first characteristic variety of a certain class of arrangements. 
Fix a hyperplane $H\in \A$, and let  $\A'=\A\setminus \{H\}$.  
A {\em pointed multinet}\/  on $\A$ consists of a multinet  on $\A$, 
together with a distinguished hyperplane $H\in \A$ such that 
$m_H>1$ and $m_H \mid n_X$ for each $X\in \XX$ such that $X\subset H$.  
Given these data, let $M'$ be the complement 
of $\A'$; then $\VV^1(M')$ has a component which is a $1$-dimensional 
subtorus of $\T_{M'}$, translated by a character of order $m_H$. 

Applying Proposition \ref{prop:tau1 trop} to this setup, we obtain 
the following result.

\begin{proposition}
\label{prop:del-sigma}
Let $\A$ be an arrangement which admits a pointed multinet, and let $\A'$ be the 
arrangement obtained from $\A$ by deleting the distinguished hyperplane $H$.  
Then 
\begin{enumerate}
\item \label{p1}
The resonance variety $\RR^1(M',\R)$
is properly contained in $\Trop(\VV^1(M'))$.
\item \label{p2}
The BNS invariant $\Sigma^1(M')$ is properly contained 
in $S(\RR^1(M',\R))^{\compl}$.
\end{enumerate}
\end{proposition}

Taking $\A$ to be the reflection arrangement of type $\operatorname{B}_3$, 
and $\A'$ to be the deleted $\operatorname{B}_3$ arrangement from \cite{Su-ta02} 
recovers the result that $\Sigma^1(M') \subsetneqq S(\RR^1(M',\R))^{\compl}$, 
proved in a different way in \cite[Example~11.8]{Su-pisa12}.  
In view of Proposition \ref{prop:del-sigma}, a more refined question to ask, 
then, is the following.

\begin{question}
\label{quest:bns-trop-arr}
For a complex hyperplane arrangement $\A$ 
with complement $M=M(\A)$, is the BNS invariant 
$\Sigma^1(M)$ equal to $S(\Trop(\VV^1(M)))^{\compl}$?
\end{question} 

\begin{ack}
We wish to thank the referee for a careful reading of the manuscript and for 
many valuable comments and suggestions that helped us to improve both 
the substance and the exposition of the paper.
\end{ack}

\newcommand{\arxiv}[1]
{\texttt{\href{http://arxiv.org/abs/#1}{arXiv:#1}}}
\newcommand{\arx}[1]
{\texttt{\href{http://arxiv.org/abs/#1}{arXiv:}}
\texttt{\href{http://arxiv.org/abs/#1}{#1}}}
\newcommand{\doi}[1]
{\texttt{\href{http://dx.doi.org/#1}{doi:#1}}}
\renewcommand{\MR}[1]
{\href{http://www.ams.org/mathscinet-getitem?mr=#1}{MR#1}}


\begin{thebibliography}{00}

\bibitem{Ar97} D.~Arapura,
{\em Geometry of cohomology support loci for local systems.
\textup{I}.}, J. Algebraic Geom. \textbf{6} (1997), no.~3, 563--597.
\MR{1487227} 

\bibitem{ACM} E.~Artal Bartolo, J.~Cogolludo, D.~Matei,
\href{https://dx.doi.org/10.2140/gt.2013.17.273}%
{\em Characteristic varieties of quasi-projective manifolds and orbifolds},
Geom. Topol. \textbf{17} (2013), no.~1, 273--309.
\MR{3035328}

\bibitem{Be92} A.~Beauville, 
\href{https://dx.doi.org/10.1007/BFb0094507}%
{\em Annulation du $H\sp 1$ pour les fibr\' es en droites plats}, in: 
Complex algebraic varieties (Bayreuth, 1990), 1--15,
Lecture Notes in Math., vol.~1507, Springer, Berlin, 1992. 
\MR{1178716} 

\bibitem{Bi07} R.~Bieri,
\href{https://dx.doi.org/10.1016/j.jpaa.2006.02.003}%
{\em Deficiency and the geometric invariants of a group}\/  
(with an appendix by P.~Schweitzer), 
J. Pure Appl. Alg. \textbf{208} (2007), no.~3, 951--959. 
\MR{2283437} 

\bibitem{BGr} R.~Bieri, J.~Groves, 
\href{https://dx.doi.org/10.1515/crll.1984.347.168}%
{\em The geometry of the set of characters induced by valuations}, 
J. Reine Angew. Math. \textbf{347} (1984), 168--195.
\MR{0733052} 

\bibitem{BNS} R.~Bieri, W.~Neumann, R.~Strebel, 
\href{https://dx.doi.org/10.1007/BF01389175}%
{\em A geometric invariant of discrete groups}, 
Invent. Math. \textbf{90} (1987), no.~3, 451--477. 
\MR{0914846}  

\bibitem{BR} R.~Bieri, B.~Renz,
\href{https://dx.doi.org/10.1007/BF02566775}%
{\em Valuations on free resolutions and higher geometric 
invariants of groups}, Comment. Math. Helvetici 
\textbf{63} (1988), no.~3, 464--497.
\MR{0960770}  

\bibitem{BK} T.~Bogart, E.~Katz, 
\href{https://dx.doi.org/10.1137/110825558}%
{\em Obstructions to lifting tropical curves in hypersurfaces}, 
SIAM J. Discrete Math. \textbf{26} (2012), no.~3, 1050--1067. 
\MR{3022123}

\bibitem{Br} K.S.~Brown,
\href{https://dx.doi.org/10.1007/BF01389176}%
{\em Trees, valuations, and the {B}ieri-{N}eumann-{S}trebel 
invariant}, Invent. Math. \textbf{90} (1987), no.~3, 479--504. 
\MR{0914847}  

\bibitem{BW12} N.~Budur, B.~Wang,
\href{https://dx.doi.org/10.1112/S0010437X14007970}%
{\em Cohomology jump loci of quasi-projective varieties},
Ann. Sci. \'{E}cole Norm. Sup. \textbf{48} (2015), no.~1.
\MR{3335842}

\bibitem{Cat91} F.~Catanese, 
\href{https://dx.doi.org/10.1007/BF01245076}%
{\em Moduli and classification of irregular Kaehler manifolds 
(and algebraic varieties) with Albanese general type fibrations}, 
Invent. Math. \textbf{104} (1991), no.~2, 263--289. 
\MR{1098610} 

\bibitem{CCM} F.~Catanese, C.~Ciliberto, M.~Mendes Lopes,
\href{https://dx.doi.org/10.1090/S0002-9947-98-01948-5}%
{\em On the classification of irregular surfaces of general type
with nonbirational bicanonical map},
Trans. Amer. Math. Soc. \textbf{350} (1998), no.~1, 275--308.
\MR{1422597}  

\bibitem{CTY} M.A.~Cueto,  E.A.~Tobis, J.~Yu, 
\href{https://dx.doi.org/10.1016/j.jsc.2010.06.011}%
{\em An implicitization challenge for binary factor analysis}, 
J. Symbolic Comput. \textbf{45} (2010), no.~12, 1296--1315. 
\MR{2733380} 

\bibitem{De08} T.~Delzant,
\href{https://dx.doi.org/10.1007/s00039-008-0679-2}%
{\em Trees, valuations, and the {G}reen-{L}azarsfeld sets}, 
Geom. Funct. Anal. \textbf{18} (2008), no.~4, 1236--1250. 
\MR{2465689}  

\bibitem{De10} T.~Delzant,
\href{https://dx.doi.org/10.1007/s00208-009-0468-8}%
{\em L'invariant de {B}ieri {N}eumann {S}trebel des groupes 
fondamentaux des vari\'{e}t\'{e}s k\"{a}hl\'{e}riennes}, 
Math. Annalen \textbf{348} (2010), no.~1, 119--125.
\MR{2657436} 

\bibitem{DeS14} G.~Denham, A.I.~Suciu,
\href{https://dx.doi.org/10.1112/plms/pdt058}%
{\em Multinets, parallel connections, and {M}ilnor fibrations  
of arrangements}, Proc. London Math. Soc. 
\textbf{108} (2014), no.~6, 1435--1470.
\MR{3218315}

\bibitem{DSY}
 G.~Denham, A.I.~Suciu, S.~Yuzvinsky, 
 \href{https://dx.doi.org/10.1007/s00029-017-0343-5}%
{\em Abelian duality and propagation of resonance},
Selecta Math. (N.S.) \textbf{23} (2017), no.~4, 2331--2367.    
\MR{3703455}

\bibitem{Di07} A.~Dimca, 
\href{https://dx.doi.org/10.4171/RLM/503}%
{\em Characteristic varieties and constructible sheaves}, 
Rend. Lincei Mat. Appl. \textbf{18} (2007), no.~4, 365--389. 
\MR{2349994} 

\bibitem{DP-ccm} A.~Dimca, \c{S}.~Papadima,
\href{https://dx.doi.org/10.1142/S0219199713500259}%
{\em Non-abelian cohomology jump loci from an analytic viewpoint}, 
Commun. Contemp. Math. \textbf{16} (2014), 
no.~4, 1350025 (47 p). 
\MR{3231055}

\bibitem{DPS-imrn} A.~Dimca, S.~Papadima, A.I.~Suciu,
\href{https://dx.doi.org/10.1093/imrn/rnm119}%
{\em Alexander polynomials: {E}ssential variables and 
multiplicities}, Int. Math. Res. Not. IMRN \textbf{2008}, 
no.~3, Art. ID rnm119, 36 pp. 
\MR{2416998}

\bibitem{DPS-duke} A.~Dimca, S.~Papadima, A.~Suciu,
 \href{https://dx.doi.org/10.1215/00127094-2009-030}%
{\em Topology and geometry of cohomology jump loci}, 
Duke Math. J. \textbf{148} (2009), no.~3, 405--457.
\MR{2527322} 

\bibitem{Dun} N.~Dunfield, 
\href{https://dx.doi.org/10.2140/pjm.2001.200.43}%
{\em Alexander and Thurston norms of fibered $3$-manifolds}, 
Pacific J. Math. \textbf{200} (2001), no.~1, 43--58. 
\MR{1863406}

\bibitem{EL97} L.~Ein, R.~Lazarsfeld, 
\href{https://dx.doi.org/10.1090/S0894-0347-97-00223-3}%
{\em Singularities of theta divisors and the birational geometry 
of irregular varieties}, J. Amer. Math. Soc. \textbf{10} (1997), 
no.~1, 243--258. 
\MR{1396893}  

\bibitem{EKL} M.~Einsiedler, M.~Kapranov, D.~Lind, 
\href{https://dx.doi.org/10.1515/CRELLE.2006.097}%
{\em Non-{A}rchimedean amoebas and tropical varieties},  
J. Reine Angew. Math. \textbf{601} (2006), 139--157. 
\MR{2289207} 

\bibitem{EN} D.~Eisenbud, W.~Neumann, 
\href{http://press.princeton.edu/titles/2356.html}%
{\em Three-dimensional link theory and invariants of plane curve
singularities}, Annals of Math. Studies, vol.~110, Princeton 
University Press, Princeton, NJ, 1985.
\MR{0817982} 

\bibitem{FY07} M.~Falk, S.~Yuzvinsky,
\href{https://doi.org/10.1112/S0010437X07002722}%
{\em Multinets, resonance varieties, and pencils of plane curves},
Compositio Math. \textbf{143} (2007), no.~4, 1069--1088.
\MR{2339840}  

\bibitem{Far} M.~Farber, 
\href{https://dx.doi.org/10.1090/surv/108}%
{\em Topology of closed one-forms}, 
Math. Surveys Monogr., vol.~108, Amer. Math. Soc., 
Providence, RI, 2004. 
\MR{2034601}  

\bibitem{FGS} M.~Farber, R.~Geoghegan, D.~Sch\"{u}tz, 
\href{https://dx.doi.org/10.1070/RM2010v065n01ABEH004663}%
{\em Closed $1$-forms in topology and geometric group theory},
Russian Math. Surveys \textbf{65} (2010), no.~1, 143--172.
\MR{2655245} 

\bibitem{FSY} M.~Farber, A.~Suciu, S.~Yuzvinsky, 
\href{https://dx.doi.org/10.4171/OWR/2007/40}%
{\em Mini-Workshop: Topology of closed one-forms and 
cohomology jumping loci}, Oberwolfach Reports \textbf{4} 
(2007), no.~3, 2321--2360.
\MR{2432117}

\bibitem{FT20} S.~Friedl, S.~Tillmann, 
\href{https://doi.org/10.5802/aif.3325}%
{\em Two generator one-relator groups and marked polytopes}, 
Ann. Inst. Fourier \textbf{70} (2020), no.~2, 831--879.
\MR{4105952}

\bibitem{FV19} S.~Friedl, S.~Vidussi, 
\href{https://doi.org/10.1017/nmj.2019.32}%
{\em Virtual algebraic fibrations of K\"{a}hler groups}, 
Nagoya Math. J. (2019). 

\bibitem{GL87} M.~Green, R.~Lazarsfeld, 
\href{https://dx.doi.org/10.1007/BF01388711}%
{\em Deformation theory, generic vanishing theorems
and some conjectures of Enriques, Catanese and Beauville}, 
Invent. Math. \textbf{90} (1987), no.~2, 389--407.
\MR{0910207}  

\bibitem{GL91} M.~Green, R.~Lazarsfeld, 
\href{https://dx.doi.org/10.2307/2939255}%
{\em Higher obstructions to deforming cohomology groups 
of line bundles}, J. Amer. Math. Soc. \textbf{4} (1991), 
no.~1, 87--103.
\MR{1076513}

\bibitem{Hi97} E.~Hironaka,
\href{https://doi.org/10.5802/aif.1573}%
{\em Alexander stratifications of character varieties}, Ann. Inst. 
Fourier (Grenoble)  \textbf{47} (1997), no.~2, 555--583.
\MR{1450425} 

\bibitem{JW}  M.~Jankins, W.~Neumann, 
\href{https://www.math.columbia.edu/~neumann/preprints/neumann_lectures on seifert manifolds.pdf}%
{\em Lectures on Seifert manifolds}, Brandeis Lecture Notes, vol.~2. 
Brandeis University, Waltham, MA, 1983. 
\MR{0741334} 

\bibitem{Ki19}  D.~Kielak, 
\href{https://doi.org/10.1007/s00222-019-00919-9}%
{\em The Bieri--Neumann--Strebel invariants via Newton polytopes}, 
Invent. Math. \textbf{219} (2019), no.~3, 1009--1068.
\MR{4055183}

\bibitem{KMM} N.~Koban, J.~McCammond, J.~Meier, 
\href{https://dx.doi.org/10.4171/GGD/323}%
{\em The BNS-invariant for the pure braid groups}, 
Groups, Geometry, and Dynamics \textbf{9} (2015), 
no.~3, 665--682.
\MR{3420539}
 
\bibitem{KP15} T.~Kohno, A.~Pajitnov, 
\href{https://doi.org/10.1515/forum-2013-0032}%
{\em Circle-valued {M}orse theory for complex hyperplane 
arrangements}, Forum Math. \textbf{27} (2015), no.~4, 2113--2128. 
\MR{3365791}

\bibitem{Lg}  D.G.~Long, 
{\em The Laurent norm},  preprint (2008), \arxiv{0808.1058}.

\bibitem{MS}  D.~Maclagan, B.~Sturmfels, 
\href{http://bookstore.ams.org/gsm-161/}%
{\em Introduction to tropical geometry}, 
Grad. Stud. Math., vol.~161, Amer. Math. Soc., 
Providence, RI, 2015.
\MR{3287221}

\bibitem{McM} C.T.~McMullen,
\href{https://dx.doi.org/10.1016/S0012-9593(02)01086-8}%
{\em The Alexander polynomial of a $3$-manifold and the 
Thurston norm on cohomology}, Ann. Sci. \'{E}cole Norm. Sup. 
\textbf{35} (2002), no.~2, 153--171. 
\MR{1914929} 

\bibitem{MMV} J.~Meier, H.~Meinert, L.~VanWyk, 
\href{https://dx.doi.org/10.1007/s000140050044}%
{\em Higher generation subgroup sets and the 
$\Sigma$-invariants of graph groups}, Comment. 
Math. Helv. \textbf{73} (1998), no.~1, 22--44.
\MR{1610579}

\bibitem{Mi62} J.~Milnor, 
\href{https://dx.doi.org/10.2307/1970268}%
{\em A duality theorem for Reidemeister torsion}, 
Ann. of Math. \textbf{76} (1962), no.~1, 137--147.
\MR{0141115 }

\bibitem{NeumannRaymond} W.~Neumann, F.~Raymond,
\href{https://doi.org/10.1007/BFb0061699}%
{\em Seifert manifolds, plumbing, $\mu$-invariant and orientation reversing 
maps}, in: Algebraic and geometric topology (Proc. Sympos., Univ. California, 
Santa Barbara, Calif., 1977), pp. 163--196, Lecture Notes in Math., vol.~664, 
Springer, Berlin, 1978.
\MR{0518415}

\bibitem{Novikov} S.P.~Novikov, 
\href{http://mi.mathnet.ru/eng/dan44681}%
{\em Multivalued functions and functionals. An analogue of the Morse
theory}, Dokl. Akad. Nauk SSSR \textbf{260} (1981), no.~1, 31--35.
\MR{0630459}

\bibitem{OS} B.~Osserman, S.~Payne, 
\href{https://www.elibm.org/article/10000294}%
{\em Lifting tropical intersections}, 
Doc. Math. \textbf{18} (2013), 121--175. 
\MR{3064984}

\bibitem{PS-adv} S.~Papadima, A.I.~Suciu,
\href{https://dx.doi.org/10.1016/j.aim.2008.09.008}%
{\em Toric complexes and Artin kernels}, 
Adv. Math. \textbf{220} (2009), no.~2, 441--477. 
\MR{2466422}

\bibitem{PS-plms} S.~Papadima, A.I.~Suciu,
\href{https://dx.doi.org/10.1112/plms/pdp045}%
{\em Bieri--{N}eumann--{S}trebel--{R}enz invariants and 
homology jumping loci}, Proc. Lond. Math.~Soc. 
\textbf{100} (2010), no.~3, 795--834.
\MR{2640291}  

\bibitem{PS-mrl} S.~Papadima, A.I.~Suciu,
\href{http://dx.doi.org/10.4310/MRL.2014.v21.n4.a13}%
{\em Jump loci in the equivariant spectral sequence}, 
Math. Res. Lett. \textbf{21} (2014), no.~4, 863--883.
\MR{3275650} 

\bibitem{PeY} J.~Pereira, S.~Yuzvinsky, 
\href{https://doi.org/10.1016/j.aim.2008.05.014}%
{\em Completely reducible hypersurfaces in a pencil}, 
Adv.  Math.  \textbf{219} (2008), no.~2, 672--688.
\MR{2435653}

\bibitem{Pa09} S.~Payne,
\href{https://doi.org/10.1007/s00209-008-0374-x}%
{\em Fibers of tropicalization}, Math. Z. \textbf{262} (2009), 
no.~2, 301--311.
\MR{2504879}

\bibitem{Ra12} J.~Rabinoff,
\href{https://doi.org/10.1016/j.aim.2012.02.003}%
{\em Tropical analytic geometry, Newton polygons, and 
tropical intersections}, Adv. Math. \textbf{229} (2012), 
no.~6, 3192--3255.
\MR{2900439}

\bibitem{Scott} P.~Scott, 
\href{https://dx.doi.org/10.1112/blms/15.5.401}%
{\emph{The geometries of {$3$}-manifolds}}, Bull. 
London Math. Soc. \textbf{15} (1983), no.~5, 401--487.
\MR{705527}

\bibitem{Sk87}  J.-Cl. Sikorav,
\href{http://perso.ens-lyon.fr/jean-claude.sikorav/textes/Novikov1987.pdf}%
{\em Homologie de {N}ovikov associ\'{e}e \`{a} une classe 
de cohomologie r\'{e}ele de degr\'{e} un}, Th\`{e}se d'Etat, 
Universit\'{e} Paris-Sud, Orsay (1987).

\bibitem{Si93}  C.~Simpson, 
\href{https://doi.org/10.24033/asens.1675 }%
{\em Subspaces of moduli spaces of rank one local systems}, 
Ann. Sci. \'{E}cole Norm. Sup. \textbf{26} (1993), no.~3, 361--401.
\MR{1222278}  

\bibitem{Su-ta02} A.I.~Suciu,
\href{https://dx.doi.org/10.1016/S0166-8641(01)00052-9}%
{\em Translated tori in the characteristic varieties of
complex hyperplane arrangements}, Topology Appl. 
\textbf{118} (2002), no.~1-2, 209--223.  
\MR{1877726}  

\bibitem{Su-pisa12} A.I.~Suciu,
\href{https://dx.doi.org/10.1007/978-88-7642-431-1_21}%
{\em Geometric and homological finiteness in free 
abelian covers}, in: {\em Configuration Spaces: Geometry, 
Combinatorics and Topology (Centro De Giorgi, 2010)}, 461--501, 
Publications of the Scuola Normale Superiore, vol.~14, 
Edizioni della Normale, Pisa, 2012. 
\MR{3203652}

\bibitem{Su-imrn} A.I.~Suciu,
\href{https://dx.doi.org/10.1093/imrn/rns246}%
{\em Characteristic varieties and Betti numbers of free 
abelian covers}, Intern. Math. Res. Notices \textbf{2014} 
(2014), no. 4, 1063--1124.
\MR{3168402}.

\bibitem{Su-tcone}  A.I.~Suciu,
\href{https://dx.doi.org/10.1007/978-3-319-31580-5_1}%
{\em Around the tangent cone theorem}, in: 
{\em Configuration Spaces: Geometry, Topology and Representation Theory}, 
1--39, Springer INdAM series, vol.~14, Springer, Cham, 2016. 
\MR{3615726}. 

\bibitem{Su-tc3d} A.I.~Suciu,
\href{https://doi.org/10.1007/s00229-020-01264-5}%
{\em Cohomology jump loci of $3$-dimensional manifolds}, 
Manuscripta Math. (2020).   

\bibitem{SYZ-jpaa} A.I.~Suciu, Y.~Yang, G.~Zhao, 
\href{https://dx.doi.org/10.1016/j.jpaa.2012.06.025}%
{\em  Intersections of translated algebraic subtori}, 
J. Pure Appl. Algebra \textbf{217} (2013), no.~3, 481--494.
\MR{2974227}

\bibitem{SYZ-pisa} A.I.~Suciu, Y.~Yang, G.~Zhao, 
\href{https://dx.doi.org/10.2422/2036-2145.201205_008}%
{\em  Homological finiteness of abelian covers}, 
Annali della Scuola Normale Superiore di Pisa \textbf{14} (2015), 
no.~1, 101--153.
\MR{3379489}

\bibitem{Th} W. P. Thurston,
\href{https://dx.doi.orgs/10.1090/memo/0339}%
{\em A norm for the homology of $3$-manifolds},
Mem. Amer. Math. Soc., \textbf{59} (1986), no.~339.
\MR{0823443}

\bibitem{Tu75} V.G.~Turaev,
\href{https://dx.doi.org/10.1070/SM1975v026n03ABEH002483}%
{\em The Alexander polynomial of a three-dimensional manifold}, 
Mat. Sb. (N.S.) \textbf{97(139)} (1975), no.~3(7), 341--359. 
\MR{0383425} 

\bibitem{Tu} V.G.~Turaev, 
\href{https://dx.doi.orgs/10.1007/978-3-0348-7999-6}%
{\em Torsions of $3$-dimensional manifolds}, 
Progress in Mathematics, vol.~208. Birkh\"{a}user 
Verlag, Basel, 2002. 
\MR{1958479}

\bibitem{Wa16} B.~Wang, 
\href{https://dx.doi.org/10.4310/MRL.2016.v23.n2.a12}%
{\em Torsion points on the cohomology jump loci of compact K\"{a}hler 
manifolds}, Math. Res. Lett. \textbf{23} (2016), no.~2, 545--563.
\MR{3512898}

\bibitem{Weibel}  C.~Weibel,
\href{http://dx.doi.org/10.1017/CBO9781139644136}%
{\em An introduction to homological algebra}, 
Cambridge Studies in Advanced Mathematics, vol.~38. 
Cambridge University Press, Cambridge, 1994.
\MR{1269324}

\bibitem{Yu} S.~Yuzvinsky, 
\href{https://dx.doi.org/10.1090/S0002-9939-08-09753-0}%
{\em A new bound on the number of special fibers in a 
pencil of curves},  Proc. Amer. Math. Soc. \textbf{137} (2009), 
no.~5, 1641--1648. 
\MR{2470822} 

\bibitem{Ziegler} G.~Ziegler, 
\href{https://dx.doi.org/10.1007/978-1-4613-8431-1}%
{\em Lectures on polytopes}. Grad. Texts in Math., vol.~152. 
Springer-Verlag, New York, 1995.
\MR{1311028}

\end{thebibliography}
\end{document}